\documentclass{amsart}
\usepackage{amssymb}
\usepackage{amscd}
\usepackage[latin1]{inputenc}
\usepackage{amsmath}
\usepackage{amsfonts}
\usepackage{latexsym}
\usepackage{verbatim}
\usepackage{graphicx}
\usepackage{epsfig}
\newtheorem{theorem}{Theorem}[section]
\newtheorem{corollary}[theorem]{Corollary}
\newtheorem{lemma}[theorem]{Lemma}
\newtheorem{question}[theorem]{Question}
\newtheorem{proposition}[theorem]{Proposition}

\theoremstyle{definition}
\newtheorem{definition}[theorem]{Definition}
\newtheorem{example}[theorem]{Example}
\newtheorem{remark}[theorem]{Remark}

\begin{document}

%\SetWatermarkScale{3}
%\SetWatermarkLightness{1.0}

\title[On non-unif. spec. and uniqueness of equilibrium state in exp. systems]{On non-uniform specification and uniqueness of the equilibrium state in expansive systems}

\begin{abstract}
In \cite{bowen}, Bowen showed that for an expansive system $(X,T)$ with specification and a potential $\phi$ with the Bowen property, the equilibrium state is unique and fully supported. %Each of those hypotheses involves a constant; one for the distance at which length-$n$ orbit segments may be shadowed using the specification property, and the other for the upper bound on variation on $n$th partial sums of $\phi$ over small $d_n$-balls.
We generalize that result by showing that the same conclusion holds for non-uniform versions of Bowen's hypotheses in which constant parameters are replaced by any increasing unbounded functions $f(n)$ and $g(n)$ with sublogarithmic growth (in $n$). 

We prove results for two weakenings of specification; the first is non-uniform specification, based on a definition of Marcus in (\cite{marcusmonat}), and the second is a significantly weaker property which we call non-uniform transitivity. %, is significantly weaker in that it only allows for shadowing two orbit segments, and yields a single gap distance rather than a bound above which all gaps are possible. (REWRITE) 
We prove uniqueness of the equilibrium state in the former case under the assumption that $\liminf_{n \rightarrow \infty} (f(n) + g(n))/\ln n = 0$, and in the latter case when $\lim_{n \rightarrow \infty} (f(n) + g(n))/\ln n = 0$. In the former case, we also prove that the unique equilibrium state has the K-property.

It is known that when $f(n)/\ln n$ or $g(n)/\ln n$ is bounded from below, equilibrium states may not be unique, and so 
%Examples from previous work are presented which show that when $f(n)/\ln n$ or $g(n)/\ln n$ is bounded from below, an expansive system may have multiple equilibrium states. Therefore, 
this work shows that logarithmic growth is in fact the optimal transition point below which uniqueness is guaranteed. Finally, we present some examples for which our results yield the first known proof of uniqueness of equilibrium state.

\end{abstract}

\date{}
\author{Ronnie Pavlov}
\address{Ronnie Pavlov\\
Department of Mathematics\\
University of Denver\\
2280 S. Vine St.\\
Denver, CO 80208}
\email{rpavlov@du.edu}
\urladdr{www.math.du.edu/$\sim$rpavlov/}
\thanks{The author gratefully acknowledges the support of NSF grant DMS-1500685.}
\renewcommand{\subjclassname}{MSC 2010}
%below are the definitions for some subject classification numbers
%22D40 Ergodic theory on groups 
%37A05 Measure-preserving transformations 
%37A15 General groups of measure preserving transformations 
%37A35 Entropy and other invariants, isomorphism, classification 
%37B10 symbolic dynamics
%37B40 topological entropy
%37B50 multi-dimensional shifts of finite type, tiling systems 
%37C40 smooth ergodic theory, invariant measures
%37C45 dimension theory of dynamical systems
%37C85 Dynamics of group actions other than Z and R, and foliations 
%37C99 smooth dynamical systems, general theory
%37D35 thermodynamic formalism, variational principles, equilibrium states
\maketitle

%\SetWatermarkScale{4}

\section{Introduction}\label{intro}

%HOW MUCH GENERAL DISCUSSION ABOUT SUBSHIFTS/ENTROPY IS NECESSARY?

%IDEAS FOR MORE: IF DEF OF NON-UNIF. SPEC INCLUDES EXISTENCE OF PERIODIC POINT, THEN GET UNIQUE MME LIM OF PER PTS AUTOMATICALLY BY 
%$g(n) = o(n)$. IF GAP DEPENDS ONLY ON MIDDLE WORD (i.e. take a min rather than max in 2-sided def), PROB GET BOUNDS ON MEASURES AND ALL %KINDS OF GOOD STUFF

%FIX HOFBAUER REFERENCE/FRAMING... IDEA FOR THIS: LEAVE AS IS, NOTE IN REMARK THAT AN EXAMINATION OF PROOF SHOWS THAT $g(n)$ COULD JUST BE $1/6 \log n$, BUT WE WANTED PRESERVATION UNDER PRODUCTS TO USE LEDRAPPIER TRICK. 

A central question in the theory of topological pressure is knowing when a dynamical system $(X,T)$ and potential $\phi: X \rightarrow \mathbb{R}$ admit a unique equilibrium state. Often, such proofs  use as hypotheses a mixing/shadowing property for the system $(X,T)$ and a regularity condition on the potential $\phi$.

One of the first and most important results of this type was proved by Bowen in \cite{bowen}, using the hypotheses of expansiveness and specification on $(X,T)$ and the so-called Bowen property on 
$\phi$. Informally, $(X,T)$ is expansive if there exists a fixed distance $\delta$ so that any unequal points of $X$ will be separated by distance at least $\delta$ under some iterate of $T$. 
Specification is the ability, given arbitrarily many orbit segments, to find a periodic point of $X$ whose orbit ``shadows'' (meaning it stays within some small distance of) those orbit segments, with gaps dependent only on the desired shadowing distance. The Bowen property is simply boundedness (w.r.t. $n$) of the differences of the partial sums $S_n \phi(x) = \phi(x) + \phi(Tx) + \ldots + \phi(T^{n-1}x)$ over pairs $(x,y)$ whose first $n$ iterates under $T$ stay within some predetermined distance. (See Section~\ref{defs} for formal definitions.) Bowen's theorem can then be stated as follows.

\begin{theorem}{\rm (\cite{bowen})}\label{bowenthm}
If $(X,T)$ is an expansive system with specification and $\phi$ is a Bowen potential, then $(X,T)$ has a unique equilibrium state for $\phi$, which is fully supported.
\end{theorem}

The assumptions of specification for $(X,T)$ and the Bowen property for $\phi$ each have associated constant bounds independent of a parameter $n$; for specification there is the bound $f(n)$ on the gap size required between shadowing orbit segments of length $n$, and for the Bowen property there is the bound $g(n)$ on the associated variation of the $n$th partial sum.

The main results of this work show that the same conclusions hold even for unbounded
$f(n)$ and $g(n)$, as long as they grow sublogarithmically with $n$. We have results using two different versions of specification; the first is called non-uniform specification, and the second, much weaker, property is called non-uniform transitivity. %Unsurprisingly, the latter requires stronger assumptions on $f$ and $g$.

\begin{theorem}\label{mainthm}
If $(X,T)$ is an expansive dynamical system (with expansivity constant $\delta$) with non-uniform specification with gap bounds $f(n)$ (at scale $\delta$), $\phi$ is a potential with partial sum variation bounds $g(n)$ (at scale $\delta$), and $\liminf_{n \rightarrow \infty} \frac{f(n) + g(n)}{\ln n} = 0$, then $X$ has a unique equilibrium state for $\phi$, which is fully supported.
\end{theorem}

\begin{theorem}\label{mainthm2}
If $(X,T)$ is an expansive dynamical system (with expansivity constant $\delta$) with non-uniform transitivity with gap bounds $f(n)$ (at scale $\delta$), $\phi$ is a potential with partial sum variation bounds $g(n)$ (at scale $\delta$), and if $\lim_{n \rightarrow \infty} \frac{f(n) + g(n)}{\ln n} = 0$, then $X$ has a unique equilibrium state for $\phi$, which is fully supported.
\end{theorem}

\begin{remark}
Though our hypotheses for these results are stated for scale $\delta$ equal to the expansivity constant, standard arguments (here given as Lemmas~\ref{gapscaleindep} and \ref{sumscaleindep}) show that they are equivalent at any scale less than $\delta$.
\end{remark}

\begin{remark}
For non-invertible surjective $(X,T)$, there is a canonical way to create an invertible system $(X', T')$ called the natural extension. It's well-known that the natural extension has the same simplex of invariant measures as that of the original system, and so Theorems~\ref{mainthm} and \ref{mainthm2} can also be applied to any $(X,T)$ whose natural extension satisfies their hypotheses. In particular, we note that whenever $(X,T)$ is positively expansive (see Theorem 2.2.32(3) of \cite{aokibook}), its natural extension is expansive. 
\end{remark}

We also prove results about preservation of these properties under expansive factors and products, which are unavoidably a bit technical, and so we postpone formal statements to Section~\ref{scaleindep}.
%The class of $(X,T,\phi)$ satisfying the hypotheses of Theorem~\ref{mainthm} is closed under products (for a natural choice of potential on the product). (THIS IS FALSE) Factors are more delicate since there's no immediate meaning for the image of $\phi$ under a factor map. The simplest result we can state is the following: the class of $(X,T)$ satisfying the hypotheses of either Theorem~\ref{mainthm} or Theorem \ref{mainthm2} when $g = 0$ is closed under expansive factors; see Section~\ref{scaleindep} for formal statements. 

In particular, the preservation under products that we prove allows us to use an argument of Ledrappier (\cite{ledrappier}) to prove strong properties of the equilibrium state from Theorem~\ref{mainthm}.

\begin{corollary}\label{Kcor}
For any $(X,T)$ and $\phi$ satisfying the hypotheses of Theorem~\ref{mainthm} and associated unique equilibrium state $\mu$, $(X,T,\mu)$ is a K-system.
\end{corollary}

Since K-systems have positive entropy, this also answers a question of Climenhaga from \cite{vaughntowers} about so-called hyperbolic potentials. Following \cite{IRRL}, a potential $\phi$ is said to be hyperbolic for $(X,T)$ if every equilibrium state has positive entropy. Climenhaga's question was the following:

\begin{question}[\cite{vaughntowers}, Question 3.20]
Is there an axiomatic condition on a subshift $(X,T)$, weaker than specification
(perhaps some form of non-uniform specification), guaranteeing that every H\"{o}lder potential
on $(X,T)$ is hyperbolic? Is there such a condition that is preserved under passing to (subshift) factors?
\end{question}

Since K-systems have positive entropy and H\"{o}lder potentials on subshifts are Bowen (see, for instance, \cite{CT1}), Corollary~\ref{Kcor} gives the following positive answer.

\begin{corollary}\label{hypcor}
The class of subshifts $(X,T)$ with non-uniform specification with gap bounds $f(n)$ satisfying $\liminf_{n \rightarrow \infty} f(n)/\ln n = 0$ is closed under (subshift) factors, and for any such subshift, every H\"{o}lder potential is hyperbolic.
\end{corollary}

We also collect results from the literature which demonstrate that when one of $f,g$ is $0$ and the other quantity grows logarithmically, uniqueness of the equilibrium state is still not guaranteed; this shows that our hypotheses cannot be weakened by too much. 

Full shifts correspond to the case $f = 0$, and for those there is the following example, based on a classical example of Hofbauer from \cite{hofbauer}.

\begin{theorem}[\cite{CL}]\label{negpotthm}
For every $\epsilon > 0$, there exists a potential $\phi$ on the full shift $(X,T)$ on $\{0,1\}$ with partial sum variation bounds $g(n) < (1 + \epsilon) \ln n$ where $(X,T)$ has multiple equilibrium states for $\phi$, whose supports are disjoint. 
\end{theorem}

\begin{remark}
The example in \cite{CL}, the so-called Double Hofbauer model, was actually for one-sided full shifts, rather than the two-sided ones treated here. Briefly, they define the potential $\phi$ in terms of the largest nonnegative integer $n$ where $x(0) = x(1) \ldots = x(n)$. It is not hard to adapt this to a two-sided version, which satisfies Theorem~\ref{negpotthm}, by instead choosing maximal $n$ for which $x(-n) = \ldots = x(n)$. In fact, this is essentially the idea behind our later Example~\ref{potex}.
\end{remark}

The case $g = 0$ corresponds to equilibrium states for constant $\phi$, i.e. measures of maximal entropy. There, we have the following result, proved independently in \cite{kwietniaketal} and \cite{pavlovspec}.

\begin{theorem}[\cite{kwietniaketal}, Lemma 6; \cite{pavlovspec}, Theorem 1.1]\label{oldnegthm}
For any positive increasing $f$ with $\liminf_{n \rightarrow \infty} \frac{f(n)}{\ln n} > 0$, there exists a subshift $(X,T)$ with non-uniform specification with gap bounds $f(n)$ and multiple ergodic measures of maximal entropy whose supports are disjoint.
\end{theorem}

\begin{remark}
Theorems~\ref{mainthm} and \ref{oldnegthm} completely answer the question of when non-uniform specification forces uniqueness of the measure of maximal entropy for expansive systems; this is the case if and only if $\liminf_{n \rightarrow \infty} \frac{f(n)}{\ln n} = 0$. 
\end{remark}

\begin{remark}
The reader may notice that there is a gap in the results we've presented for full shifts; Theorem~\ref{negpotthm} shows that non-uniqueness can happen for $g(n)$ with $\lim_{n \rightarrow \infty} \frac{g(n)}{\ln n}$ arbitrarily close to $1$, and Theorems~\ref{mainthm} and \ref{mainthm2} guarantee uniqueness only when $\frac{g(n)}{\ln n}$ approaches $0$ in general or along a subsequence. 
In fact, a careful reading of our proofs shows that these theorems hold as long as the limit or liminf of $\frac{g(n)}{\ln n}$ is smaller than $\frac{1}{6}$.

We chose not to use this as our hypothesis both due to an aesthetic preference for $f(n)$ and $g(n)$ to be on equal footing, and because we wanted hypotheses invariant under products in order to use Ledrappier's arguments to prove the K-property for the unique equilibrium state. However, it seems that the actual ``transition point'' for $g(n)$ is likely of the form $C \log n$ for $\frac{1}{6} \leq C \leq 1$, and it would be an interesting technical problem to prove this and find the $C$ in question; we will not, however, treat that question in this work though.
\end{remark}

%(MENTION HOLDER RESULTS FROM WALTERS THAT SHOW THAT VERY SMALL VARIATION IMPLIES UNIQUE MME WITH MORE?)

%(MENTION SUBTLETY OF WHETHER THERE ARE NON-UNIFORM TRANS COUNTEREXAMPLES FOR SMALLER f?)

%(IS THIS PORTION REALLY APPROPRIATE/NECESSARY?)
The techniques used to prove Theorems~\ref{mainthm} and \ref{mainthm2} are somewhat similar to the proofs of previous weaker results from \cite{pavlovspec}, which applied only to measures of maximal entropy on subshifts and showed only that two such measures $\mu$, $\nu$ could not have disjoint supports. Roughly speaking, the proof in \cite{pavlovspec} involved combining ``large'' collections of words based on $\mu$ and $\nu$ to create more words in $\mathcal{L}(X)$ than there should be by definition of $h(X,T)$. The assumption of disjoint supports of $\mu$ and $\nu$ implied that the words from the two collections could not have overlap above a certain length, which ensured that all words created were distinct. 

For the results in this work, several changes must be made. First of all, obviously `words' must be replaced with `orbit segments shadowed by a very small distance' for general expansive systems, and `number of words' must be replaced by `partition function for an $(n,\delta)$-separated set' for general potentials. These changes require some technical results about changes of scale for expansive systems (see Section~\ref{scaleindep}), but the ideas are all essentially present in previous work of Bowen and others. 

The main advance in this work is dealing with $\mu \neq \nu$ whose supports may not be disjoint. The best we can do then is to assume ergodicity of $\mu, \nu$, which implies their mutual singularity, and therefore the existence of %disjoint measurable sets $A, B$ with $\mu(A) = \nu(B) = 0$, and therefore of 
disjoint compact sets $C, D$ with $\mu(C), \nu(D)$ arbitrarily small (and positive distance $d(C,D)$). The new idea here is the use of the maximal ergodic theorem, which allows us to define ``large'' collections of orbit segments based on $\mu$ and $\nu$ where all initial segments of $\nu$-orbit segments have a large proportion of visits to $C$, and all terminal segments of $\mu$-orbit segments  have a large proportion of visits to $D$. This means that initial segments of $\nu$-orbit segments and terminal segments of $\mu$-orbit segments are separated by $d(C,D)$ at some point, mimicking the lack of overlap from the proof in \cite{pavlovspec}. This is enough to create an $(n, \delta)$-separated collection by combining segments from the two collections whose partition function is larger than the pressure $P(X, T, \phi)$ should allow, achieving the desired contradiction.

\begin{remark} 
In various works (including \cite{CT1}, \cite{CT2}, and \cite{CT3}), Climenhaga and Thompson have defined different weakenings of the specification property, which allowed them to both generalize Bowen's results in a different direction, even treating some non-expansive systems and continuous flows. Without going into full detail here, their definitions involve decomposing all orbit segments of points in the system into prefixes, cores, and suffixes, where the sets of possible prefixes/suffixes are ``small'' in some sense, and where for any $N$, the collection of segments whose prefix and suffix are shorter than $N$ has specification (in the sense that one can always find a point which shadows arbitrarily many such segments, of any lengths, with constant gaps). Existence of such a Climenhaga-Thompson decomposition is less restrictive than non-uniform specification in that specification properties must hold only for a subset of orbit segments, but is more restrictive in that the property required for that subset (weak specification) is significantly stronger. To our knowledge, neither of non-uniform specification or a Climenhaga-Thompson decomposition implies the other.
\end{remark}

Finally, we summarize the structure of the paper: Section~\ref{defs} contains relevant definitions and background on topological dynamics, ergodic theory, and thermodynamic formalism. Section~\ref{scaleindep} contains some results about preservation of various hypotheses under changes of scale, expansive factors, and products. Section~\ref{proofs} contains the proofs of Theorems~\ref{mainthm} and \ref{mainthm2}, including various auxiliary results. Finally, Section~\ref{examples} contains some examples for which our results imply uniqueness of equilibrium state and for which we believe this to be previously not known.

\section*{acknowledgments} 
The author would like to thank Jerome Buzzi and Sylvain Crovisier for pointing out that the proof of Theorem~\ref{mainthm} could be easily adapted to use the weaker hypothesis of non-uniform transitivity (yielding Theorem~\ref{mainthm2}), and would also like to thank Fran{\c c}ois Ledrappier for discussions about the use of the techniques from \cite{ledrappier} for proving the K-property for unique equilibrium states.

\section{Definitions}\label{defs}

\begin{definition}
A \textbf{dynamical system} is given by a pair $(X,T)$ where $X$ is a compact metric space and $T: X \rightarrow X$ is a homeomorphism.
\end{definition}

\begin{definition}
A dynamical system $(X, T)$ is \textbf{expansive} if there exists $\delta > 0$ (called an \textbf{expansivity constant}) so that for all unequal $x, y \in X$, there exists $n \in \mathbb{Z}$ for which $d(T^n x, T^n y) > \delta$.
\end{definition}

A particular class of expansive dynamical systems are given by subshifts, to which the next few definitions refer.

\begin{definition}
Given a finite set $A$ called the \textbf{alphabet}, a subshift $(X,T)$ is given by $X \subset A^{\mathbb{Z}}$ which is closed (in the product topology) and invariant under the left shift map $T$ defined by $(Tx)(i) = x(i+1)$ for $i \in \mathbb{Z}$.
\end{definition}

Every subshift is expansive; simply choose any $\delta$ so that $d(x,y) \leq \delta \Longrightarrow x(0) = y(0)$. Then, any $x \neq y \in X$ must have $x(n) \neq y(n)$ for some $n$, and then $(T^n x)(0) \neq (T^n y)(0)$, so $d(T^n x, T^n y) > \delta$.

\begin{definition}
The \textbf{language} of a subshift $(X,T)$, denoted by $\mathcal{L}(X)$, is the set of finite strings of letters from $A$ (called \textbf{words}) which appear in some $x \in X$.
\end{definition}

We now return to definitions for more general expansive systems. 

\begin{definition}
Given an expansive dynamical system $(X,T)$ and any $n \in \mathbb{N}$ and $\epsilon > 0$, a set $S$ is called \textbf{$(n, \epsilon)$-separated} if for all unequal $x, y \in S$, there exists $0 \leq k < n$ so that $d(T^k x, T^k y) > \epsilon$. 
\end{definition}

\begin{definition}
Given a continuous function $\phi: X \rightarrow \mathbb{R}$ (called a \textbf{potential}), the \textbf{partial sums} of $\phi$ are the functions $S_n \phi: X \rightarrow \mathbb{R}$ defined by $S_n \phi(x) = \sum_{i=0}^{n-1} \phi(T^i x)$.
\end{definition}

%\begin{definition}
%(DO WE USE THIS?) For a dynamical system $(X, T)$, $x \in X$ is called \textbf{periodic} if there exists $n > 0$ for which $T^n x = x$. 
%\end{definition}

\begin{definition}
For a dynamical system $(X,T)$ and potential $\phi$, the $n$th \textbf{partition function} of $\phi$ at scale $\eta$ are the functions
\[
Z(X, T, \phi, n, \eta) := \max_{S \textrm{ is } (n,\eta)-\textrm{separated}} \sum_{x \in S} e^{S_n \phi (x)}.
\]
\end{definition}

\begin{definition}
The \textbf{topological pressure at scale $\eta$} of $(X,T,\phi)$ is
\[
P(X, T, \phi, \eta) := \lim_{n \rightarrow \infty} \frac{\ln Z(X, T, \phi, n, \eta)}{n}.
\]
The \textbf{topological pressure} of $(X,T)$ for a potential $\phi$ is
\[
P(X, T, \phi) := \lim_{\eta \rightarrow 0} P(X, T, \phi, \eta).
\]
\end{definition}

\begin{lemma}[\cite{ruelle}]\label{presscaleindep}
If $(X,T)$ is expansive with expansivity constant $\delta$, then for any $\eta \leq \delta$, $P(X, T, \phi) = P(X, T, \phi, \eta)$.
\end{lemma}

We also need some definitions from measure-theoretic dynamics. All measures considered in this paper will be $T$-invariant Borel probability measures on $X$ for $(X,T)$ an expansive dynamical system, and we denote the space of such measures by $\mathcal{M}(X,T)$. 

\begin{definition}
A measure $\mu$ on $A^{\mathbb{Z}}$ is {\bf ergodic} if any measurable set $C$ which is invariant, i.e. $\mu(C \bigtriangleup TC) = 0$, has measure $0$ or $1$. 
\end{definition}

Not all $T$-invariant measures are ergodic, but a well-known result called the ergodic decomposition shows that any non-ergodic measure can be written as a ``weighted average'' (formally, an integral) of ergodic measures. Also, whenever ergodic measures $\mu$ and $\nu$ are unequal, in fact they must be mutually singular (written $\mu \perp \nu$), i.e. there must exist a set $R$ with $\mu(R) = \nu(R^c) = 0$. (See Chapter 6 of \cite{walters} for proofs and more information.)

When a measure $\mu$ is ergodic and $f \in L^1(\mu)$, the ergodic averages $\displaystyle \frac{1}{N} \sum_{i=0}^{N-1} f(T^i x)$ converge $\mu$-a.e. to the ``correct'' value $\int f \ d\mu$; this is essentially the content of Birkhoff's ergodic theorem. We will need the following related result, which deals with the supremum of such averages rather than their limit.

\begin{theorem}[\cite{maximal}, Theorem 2 (Maximal Ergodic Theorem)]\label{maxerg}
For $f \in L^1(\mu)$, define $M^+f := \sup_{N \in \mathbb{N}} \frac{1}{N} \sum_{i=0}^{N-1} f(T^i x)$. Then for any $\lambda \in \mathbb{R}$,
\[
\lambda \mu(\{M^+f(x) > \lambda\}) \leq \int_{M^+f(x) > \lambda} f \ d\mu.
\]
\end{theorem}

The following corollary is immediate.

\begin{corollary}\label{maxergcor}
For nonnegative $f \in L^1(\mu)$ and $Mf$ as in Theorem~\ref{maxerg}, and any $\lambda \in \mathbb{R}$,
\[
\mu(\{M^+f(x) \leq \lambda\}) \geq 1 - \frac{\|f\|_1}{\lambda}.
\]
\end{corollary}

We note that by considering $T^{-1}$ instead, both of these results also hold when $M^+f$ is replaced by $M^-f := \sup_{N \in \mathbb{N}} \frac{1}{N} \sum_{i=0}^{N-1} f(T^{-i} x)$.

We also need concepts from measure-theoretic entropy/pressure; for more information/proofs, see \cite{walters}.

\begin{definition}\label{ent1}
For any $\mu \in \mathcal{M}(X,T)$, and finite measurable partition $\mathcal{P}$ of $X$, the \textbf{information of $\mathcal{P}$ with respect to $(X,T,\mu)$} is
\[
H(X, T, \mu, \mathcal{P}) := \sum_{A \in \mathcal{P}} -\mu(A) \ln \mu(A),
\]
where terms with $\mu(A) = 0$ are omitted from the sum.
\end{definition}

\begin{definition}\label{ent2}
For any $\mu \in \mathcal{M}(X,T)$ and finite measurable partition $\mathcal{P}$ of $X$, the \textbf{entropy of $\mathcal{P}$ with respect to $(X,T,\mu)$} is
\[
h(X, T, \mu, \mathcal{P}) := \lim_{n \rightarrow \infty} \frac{H\left(X, T, \mu, \bigvee_{i=0}^{n-1} T^i \mathcal{P}\right)}{n}.
\]
\end{definition}
Note that by subadditivity, it is always true that $H\left(X, T, \mu, \bigvee_{i=0}^{n-1} T^i \mathcal{P}\right) \geq nh(X, T, \mu, \mathcal{P})$.

\begin{definition}\label{ent3}
For any $\mu \in \mathcal{M}(X,T)$, the \textbf{entropy of $(X,T,\mu)$} is
\[
h(X, T, \mu) := \sup_{\mathcal{P}} h(X, T, \mu, \mathcal{P}).
\]
\end{definition}

\begin{definition}
We say that $\mu \in \mathcal{M}(X,T)$ has the \textbf{K-property} if for every partition $\mathcal{P}$
consisting of nonempty sets, $h(X, T, \mu, \mathcal{P}) > 0$. 
\end{definition}
Note that any $\mu \in \mathcal{M}(X,T)$ with the K-property trivially has $h(X, T, \mu) > 0$.

\begin{definition}
A partition $\mathcal{P}$ is a \textbf{generating partition} for $(X,T,\mu)$ if $\bigvee_{i \in \mathbb{Z}} T^i \mathcal{P}$ separates $\mu$-a.e. points of $X$.
\end{definition}

\begin{theorem}
If $\mathcal{P}$ is a generating partition, then $h(X, T, \mu) = h(X, T, \mu, \mathcal{P})$.
\end{theorem}

By expansivity, any partition $\mathcal{P}$ of sets whose diameters are all less than $\delta$ is a generating partition for all $\mu$, and so we have the following fact:
 
\begin{lemma}
If $(X,T)$ is expansive with expansivity constant $\delta$ and $\mathcal{P}$ consists of sets whose diameters are all less than $\delta$, then $h(X, T, \mu) = h(X, T, \mu, \mathcal{P})$ for any measure $\mu \in \mathcal{M}(X,T)$.
\end{lemma}

%In Definition \ref{ent3}, a standard subadditivity argument shows that the limit can be replaced by an infimum; i.e. for any $n$,
%\begin{equation}\label{subadd} 
%h(\mu) \leq \frac{-1}{n} \sum_{A \in \mathcal{P}} \mu(A) \ln \mu(A).
%\end{equation}

The relationship between topological pressure and measure-theoretic entropy is given by the following Variational Principle:

\begin{theorem}[\cite{ruelle}]
For any dynamical system $(X,T)$ and continuous $\phi$,
\[
P(X, T, \phi) = \sup_{\mu} h(X, T, \mu) + \int \phi \ d\mu.
\]
\end{theorem}

\begin{definition}
For any $(X,T)$, an \textbf{equilibrium state} for $(X,T)$ and $\phi$ is a measure $\mu$ on $X$ for which $P(X, T, \phi) = h(X, T, \mu) + \int \phi \ d\mu$.
\end{definition}

\begin{theorem}
If $(X,T)$ is expansive, then the entropy map $\mu \mapsto h(X, T, \mu)$ is upper semi-continuous.
\end{theorem}

As a corollary, if $(X,T)$ is expansive and $\phi$ is continuous, then it has an equilibrium state; the upper semi-continuous function $\mu \mapsto h(X, T, \mu) + \int \phi \ d\mu$ must achieve its supremum $P(X, T, \phi)$ on the compact space $\mathcal{M}(X,T)$ (endowed with the weak-$*$ topology). In fact, the ergodic decomposition, along with the fact that the entropy map is affine (\cite{walters}, Theorem 8.1), implies that the extreme points of the simplex of equilibrium states are precisely the ergodic equilibrium states. In particular, any $(X,T,\phi)$ with multiple equilibrium states also has multiple ergodic equilibrium states. 

\begin{definition}[\cite{IRRL}]
A potential $\phi$ on $(X,T)$ is called \textbf{hyperbolic} if every equilibrium state $\mu$ has $h(X, T, \mu) > 0$.
\end{definition}

\begin{remark}
Though this is not the original definition from \cite{IRRL}, it was shown to be an equivalent one in their Proposition 3.1.
\end{remark}

Our remaining definitions pertain to the hypotheses used in Theorems~\ref{mainthm} and 
\ref{mainthm2}. The first relates to the potential $\phi$.

\begin{definition}
Given a dynamical system $(X,T)$ and a potential $\phi$, the \textbf{partial sum variations of $\phi$ at scale $\eta$} are given by
\[
V(X, T, \phi, n, \eta) = \max_{\{(x,y) \ : \ \forall 0 \leq i < n, \ d(T^i x, T^i y) < \eta\}} |S_n \phi(x) - S_n \phi(y)|.
\]
We say that $\phi$ has \textbf{partial sum variation bounds $g(n)$ at scale $\eta$} if $g(n) \geq V(X, T, \phi, i, \eta)$ whenever $i \leq n$.
\end{definition}

Our remaining definitions are for specification properties on $(X,T)$. We first need a general notion of shadowing.

\begin{definition}\label{shadow}
Given a dynamical system $(X,T)$, $\epsilon > 0$, points $z, x_1, \ldots, x_k \in X$, and integers 
$n_1$, $\ldots$, $n_k$, $m_1$, $\ldots$, $m_{k-1}$, we say that \textbf{$z$ $\eta$-shadows $(x_i)$ for $(n_i)$ iterates with gaps $(m_i)$} if for every 
$0 < i < k$ and $0 \leq m < n_i$, 
\[
d(T^{m + \sum_{j=1}^{i-1} (n_j + m_j)} z, T^m x_i) < \eta.
\]
\end{definition}

\begin{definition}\label{spec}
A dynamical system $(X,T)$ has \textbf{specification} if for any $\eta > 0$, there exists a constant $C(\eta)$ so that for any $k$, any points $x_1$, $\ldots$, $x_k \in X$, and any integers 
$n_1$, $\ldots$, $n_{k-1}$, $n_k$, $m_1$, $\ldots$, $m_{k}$ satisfying $m_i \geq C(\eta)$ when 
$0 < i < k$, there exists a point $z \in X$ which $\eta$-shadows $(x_i)$ for $(n_i)$ iterates with gaps $(m_i)$ and for which $T^{\sum_{j=1}^{k} (n_j + m_j)} z = z$.
\end{definition}

\begin{remark}
A related property in the literature is \textbf{weak specification}, which is identical to the definition above except that no periodicity of $z$ is assumed. For expansive $(X,T)$, this distinction is irrelevant; weak specification in fact implies specification (see \cite{kwietniaketal2}, Lemma 9).
\end{remark}

We now move to the mixing properties which we will consider in this work, both of which can be thought of as non-uniform generalizations of specification with no assumption of periodicity. 

\begin{definition}\label{nonunifdef}
For an increasing function $f: \mathbb{N} \rightarrow \mathbb{N}$, a dynamical system $(X,T)$ has \textbf{non-uniform specification with gap bounds $f(n)$ at scale $\eta$} if for any $k$, any points $x_1$, $x_2$, $\ldots$, $x_k \in X$, and any integers $n_1$, $\ldots$, $n_{k-1}$, $n_k$, $m_1$, $\ldots$, $m_{k-1}$ satisfying $m_i \geq \max(f(n_i), f(n_{i+1}))$, there exists a point $z \in X$ which $\eta$-shadows $(x_i)$ for $(n_i)$ iterates with gaps $(m_i)$.
\end{definition}

\begin{remark}
This property is almost the same as the main property used by Marcus in \cite{marcusmonat} (which was not there given a name). There are two differences: the first is that Marcus required $\frac{f(n)}{n} \rightarrow 0$ as part of his definition, and the second is that in Marcus's definition $m_i$ was only assumed greater than or equal to $f(n_i)$. Essentially, non-uniform specification only guarantees the ability to shadow when gaps are large enough in comparison to the lengths of orbit segments being shadowed before and after the gap, and Marcus's unnamed property requires only that gaps be large compared to the orbit segment before the gap.
\end{remark}

We also consider the following significantly weaker property of non-uniform transitivity, which is weaker than non-uniform specification in two important ways. The first is that it only guarantees the ability to shadow two orbit segments, rather than arbitrarily many, and the second is that it guarantees the existence of only a single gap length which allows for shadowing, rather than guaranteeing that all gaps above a certain threshold suffice. 

\begin{definition}
For an increasing function $f: \mathbb{N} \rightarrow \mathbb{N}$, a dynamical system $(X,T)$ has \textbf{non-uniform transitivity with gap bounds $f(n)$ at scale $\eta$} if for any $n$ and any points $x,y \in X$, there exists $i \leq f(n)$ and $z \in X$ which $\eta$-shadows 
$(x,y)$ for $(n,n)$ iterates with gap $i$. 
\end{definition}

\begin{remark}
We note that for $(X,T)$ with this property and any $n$ and $j,k \leq n$, one can also find $z$ which $\eta$-shadows $(x,y)$ for $(j,k)$ iterates with some gap $i \leq f(n)$. This is because one can instead start with the pair $(T^{-(n-j)} x, y)$, and for any $z$ which $\eta$-shadows $(T^{-(n-j)} x, y)$ for $(n,n)$ iterates with gap $i$, it is immediate that $T^{n-j} z$ $\eta$-shadows $(x,y)$ for $(j,k)$ iterates with gap $i$.
\end{remark}

%\begin{remark}
%For these properties, the assumption that $f(n)$ is nondecreasing is sometimes explicitly required in the literature. We have chosen not to impose that condition here largely because partial sum variations $g(n)$ for potentials are not necessarily monotone, and we'd like to keep $f$ and $g$ on somewhat equal footing. (REWRITE?)
%\end{remark}

%\begin{remark}
%In several works, including \cite{kwietniaketal} and \cite{yamamoto}, what we call non-uniform specification is referred to by other %names, such as almost weak specification or almost specification. 
%\end{remark}

\section{Preservation under factors, products, and changes of scale}\label{scaleindep}

In this section, we summarize some simple results illustrating preservation of various properties/quantities under expansive factors, products, and changes of scale.

The following results show that the hypotheses used for Theorems~\ref{mainthm} and \ref{mainthm2} do not depend on the scale used.

\begin{lemma}\label{gapscaleindep}
If $(X,T)$ is expansive (with expansivity constant $\delta$) and has non-uniform specification 
(transitivity) at scale $\delta$ with gap bounds $f(n)$, then for every $\eta < \delta$ there exists a constant $C = C(\eta)$ so that $(X,T)$ has non-uniform specification (transitivity) at scale $\eta$ with gap bounds $f(n + C) + C$.
\end{lemma}

\begin{proof}
We present only the proof for non-uniform specification, as the one for non-uniform transitivity is extremely similar. Choose $(X,T)$ as in the theorem, and any $\eta < \delta$. By expansivity, there exists $N$ so that if $d(T^i x, T^i y) < \delta$ for $-N \leq i \leq N$, then $d(x,y) < \eta$. Now, choose any $k,n \in \mathbb{N}$, any $x_1, \ldots, x_k \in X$, and any $n_1, \ldots, n_{k-1} \geq f(n + 2N) + 2N$. Use non-uniform specification to choose $y \in X$ which $\delta$-shadows $(T^{-N} x_1, \ldots, T^{-N} x_k)$ for $(n + 2N, \ldots, n + 2N)$ iterates, with gaps $(n_1 - 2N, \ldots, n_{k-1} - 2N)$. Then by definition of $N$, $T^N y$ $\eta$-shadows $(x_1, \ldots, x_k)$ for $(n, \ldots, n)$ iterates, with gaps $(n_1, \ldots, n_{k-1})$, proving the desired non-uniform specification at scale $\eta$.
\end{proof}

\begin{lemma}\label{sumscaleindep}
If $(X, T)$ is expansive (with expansivity constant $\delta$), $\eta < \delta$, and $\phi$ is a potential with partial sum variation bounds $g(n)$ at scale $\eta$, then there exists a constant $D = D(\eta)$ so that for every $n$, $\phi$ has partial sum variation bounds $g(n) + D$ at scale $\eta$.
\end{lemma}

\begin{proof}
Choose $(X,T)$ and $\phi$ as in the theorem, and any $\eta < \delta$. By expansivity, there exists $N$ so that if $d(T^i x, T^i y) < \delta$ for $-N \leq i \leq N$, then $d(x,y) < \eta$. For the second inequality, choose any $n > 2N$ and any $x,y$ for which $d(T^i x, T^i y) < \delta$ for $0 \leq i < n$. Then by definition of $N$, $d(T^i (T^N) x, T^i (T^N) y) < \eta$ for $0 \leq i < n - 2N$, and so
\begin{multline*}
|S_n \phi(x) - S_n \phi(y)| \leq 2N (\sup \phi - \inf \phi) + |S_{n - 2N} \phi(T^N x) - S_{n - 2N} \phi(T^N y)| \\
\leq 2N (\sup \phi - \inf \phi) + g(n-2N) \leq g(n) + 2N(\sup \phi - \inf \phi).
\end{multline*}

Taking the supremum over such pairs $(x,y)$ completes the proof. 

\end{proof}

%The following is then an immediate corollary of Proposition~\ref{hypfactor} and Lemma~\ref{gapscaleindep}.

\begin{proposition}\label{hypfactor}
If $f: (X,T) \rightarrow (Y,S)$ is a factor map, then for all $\epsilon$ there exists $\delta$ such that if $(X,T)$ has non-uniform specification (transitivity) with gap bounds $f(n)$ at scale $\delta$, then $(Y,S)$ has non-uniform specification (transitivity) with gap bounds $f(n)$ at scale $\epsilon$.
\end{proposition}

\begin{proof}
We give a proof only for non-uniform specification, as the proof for non-uniform transitivity is trivially similar. Given $\epsilon$, use uniform continuity of $f$ to choose $\delta$ so that $d(x,x') < \delta \Longrightarrow d(f(x), f(x')) < \epsilon$. Now, given any $y_1, \ldots, y_k \in Y$ and $n_1, \ldots, n_k$, $m_1$, $\ldots$, $m_{k-1}$, choose arbitrary $x_i \in f^{-1}(y_i)$ and use non-uniform specification of $(X,T)$ at scale $\delta$ to find $z \in X$ which $\delta$-shadows $(x_i)$ for $(n_i)$ iterates with gaps $(m_i)$. It is immediate that $f(z) \in Y$ $\epsilon$-shadows $(y_i)$ for $(n_i)$ iterates with gaps $(m_i)$, completing the proof.
\end{proof}

The following corollary follows immediately by using Lemma~\ref{gapscaleindep}.

\begin{corollary}\label{factorpres}
If $f: (X,T) \rightarrow (Y,S)$ is a factor map, $(X,T)$ and $(Y,S)$ are expansive 
(with expansivity constants $\delta$ and $\eta$ respectively), and $(X,T)$ has non-uniform specification 
(transitivity) with gap bounds $f(n)$ at scale $\delta$ where $\liminf_{n \rightarrow \infty} f(n)/\ln n = 0$ ($\lim_{n \rightarrow \infty} f(n)/\ln n = 0$), then there exists 
$\overline{f}(n)$ with $\liminf_{n \rightarrow \infty} \overline{f}(n)/\ln n = 0$ ($\lim_{n \rightarrow \infty} \overline{f}(n)/\ln n = 0$) for which $(Y,S)$ has non-uniform specification 
(transitivity) with gap bounds $\overline{f}(n)$ at scale $\eta$.
\end{corollary}

\begin{remark}
The class of $(X,T)$ satisfying either hypothesis of this corollary is then closed under expansive factors, and by Theorems~\ref{mainthm} and \ref{mainthm2}, any such system has a unique equilibrium state for any potential $\phi$ with partial sum variation bounds $g(n)$ satisfying $\lim_{n \rightarrow \infty} g(n)/\ln n = 0$.
\end{remark}

%(SAY ANYTHING ABOUT ANSWERING VAUGHN'S QUESTION ON HYPERBOLIC POTENTIALS ON SUBSHIFTS?)

%The following corollary is immediate. 

%\begin{corollary}\label{ratescaleindep}
%If $(X,T)$ is expansive (with expansivity constant $\delta$) and has non-uniform specification at scale $\delta$ with gap function and $\phi$ is a potential with partial sum variations $g(n)$ at scale $\delta$, and if $\liminf \frac{f(n) + g(n)}{\ln n} = 0$, then for every $\eta < \delta$, $(X,T)$ has non-uniform specification at scale $\eta$ with gap function $\overline{f}(n)$ and $\phi$ has partial sum variations $\overline{g}(n)$ at scale $\eta$ satisfying $\liminf \frac{\overline{f}(n) + \overline{g}(n)}{\ln n} = 0$.
%\end{corollary}

We now move to products, with the goal of proving Corollary~\ref{Kcor}. All products of metric spaces will be endowed with the $d_{\infty}$ metric defined by $d_{\infty}((x_1, y_1), (x_2, y_2)) = \max(d_1(x_1, y_1), d_2(x_2, y_2))$. The proofs of the following results are left to the reader.

\begin{lemma}\label{productpres1}
If $(X_1,T_1)$, $(X_2,T_2)$ are expansive dynamical systems with the same expansivity constant $\delta$ and which have non-uniform specification with gap bounds $f_1(n)$ and $f_2(n)$ (at scale $\delta$) respectively, then $(X_1 \times X_2, T_1 \times T_2)$ is an expansive dynamical system with expansivity constant $\delta$ which has non-uniform specification with gap bounds $f(n) = \max(f_1(n), f_2(n))$ at scale $\delta$.
\end{lemma} 

\begin{lemma}
If $(X_1, T_1)$ and $(X_2, T_2)$ are expansive dynamical systems with the same expansivity constant $\delta$ and $\phi_1$ and $\phi_2$ are potentials with partial sum variation bounds $g_1(n)$ and $g_2(n)$ (at scale $\delta$) respectively, then the potential $\phi$ on $(X_1 \times X_2, T_1 \times T_2)$ defined by $\phi(x_1, x_2) := \phi_1(x_1) + \phi_2(x_2)$ has partial sum variation bounds $g(n) = g_1(n) + g_2(n)$ at scale $\delta$.
\end{lemma}

\begin{corollary}\label{productpres2}
If $(X_1, T_1)$ and $(X_2, T_2)$ are expansive dynamical systems with non-uniform gap specification with gap bounds $f_1(n)$ and $f_2(n)$ respectively, and if $\phi_1$ and $\phi_2$ are potentials with partial sum variation bounds $g_1(n)$ and $g_2(n)$ respectively (all at scales equal to the relevant expansivity constants), and if there exists a sequence $n_k$ so that $\lim_{k \rightarrow \infty} \frac{f_1(n_k)}{\ln n_k} = \lim_{k \rightarrow \infty} \frac{f_2(n_k)}{\ln n_k} = \lim_{k \rightarrow \infty} \frac{g_1(n_k)}{\ln n_k} = \lim_{k \rightarrow \infty} \frac{g_2(n_k)}{\ln n_k} = 0$, and the potential $\phi$ on $(X_1 \times X_2, T_1 \times T_2)$ is defined by $\phi(x_1, x_2) := \phi_1(x_1) + \phi_2(x_2)$, then $(X_1 \times X_2, T_1 \times T_2, \phi)$ satisfies the hypotheses of Theorem~\ref{mainthm}.
\end{corollary}

\begin{remark}
The reason no assumption need be made on the equality of expansivity constants in Corollary~\ref{productpres2} is that one can always render them equal (say to $1$) by normalizing the metrics.
\end{remark}

%\begin{corollary}\label{selfproduct}
%If $(X,T)$ and $\phi$ satisfy the hypotheses of Theorem~\ref{mainthm}, then $(X \times X, T \times T)$ and $\phi^{(2)}$ defined by $\phi^{(2)}(x,y) := \phi(x) + \phi(y)$ does as well.
%\end{corollary}

We can now apply the following result of Ledrappier. 

\begin{theorem}[\cite{ledrappier}]\label{ktrick}
If $(X,T)$ is an expansive dynamical system, $\phi$ is a potential, and $(X \times X, T \times T)$ has a unique equilibrium state for the potential $\phi^{(2)}$ defined by $\phi^{(2)}(x,y) = \phi(x) + \phi(y)$, then $(X,T)$ has a unique equilibrium state $\mu$ for $\phi$, and $(X,T,\mu)$ is a K-system.
\end{theorem}

Corollary~\ref{Kcor} is now implied by Corollary~\ref{productpres2} and Theorem~\ref{ktrick}. Corollary~\ref{hypcor} follows as well: the closure under expansive factors comes from Corollary~\ref{factorpres}, H\"{o}lder potentials are Bowen (i.e. have bounded partial sum variation bounds) for subshifts, and positive entropy of the unique equilibrium state comes from Corollary~\ref{Kcor} since K-systems have positive entropy.

Finally, we need some technical results about behavior of separated sets/partition functions under changes of scale. The proof of the following lemma is motivated by arguments from \cite{bowen}. 

\begin{lemma}\label{sepscaleindep}
If $(X,T)$ is expansive (with expansivity constant $\delta$) and $\eta < \delta$, then there exists a constant $M = M(\eta)$ so that for every $n$, every $(n,\eta)$-separated set can be written as a disjoint union of $M$ sets which are each $(n,\delta)$-separated.
\end{lemma}

\begin{proof}
Choose $(X,T)$ and $\eta$ as in the theorem. By expansiveness, there exists $N$ so that if $d(T^i x, T^i y) < \delta$ for $-N \leq i \leq N$, then $d(x,y) < \eta$. Since $T$ is a homeomorphism, there exists $\alpha > 0$ so that if $d(x,y) < \alpha$, then $d(T^i x, T^i y) < \delta$ for $-N \leq i \leq N$. Take any partition $\mathcal{P} = \{A_i\}_{i=1}^k$ of $X$ by sets of diameter less than $\alpha$.

Now, choose any $n$ and $(n,\eta)$-separated set $S$. For each $1 \leq i,j \leq k$, define $S_{i,j} = S \cap A_i \cap T^{-n} A_j$. We claim that each $S_{i,j}$ is $(n,\delta)$-separated, which will complete the proof for $M = k^2$. To see this, fix any $i,j$ and any $x \neq y \in S_{i,j}$. Since $x,y \in S$ and $S$ is $(n,\eta)$-separated, there exists $0 \leq m < n$ so that $d(T^m x, T^m y) \geq \eta$. Since $x,y \in A_i$, $d(x,y) < \alpha$, and so $d(T^i x, T^i y) < \eta$ for $0 \leq i \leq N$. Similarly, since $x,y \in T^{-n} A_j$, $d(T^n x, T^n y) < \alpha$, and so $d(T^i x, T^i y) < \eta$ for $n - N \leq i \leq n$. Therefore, $N < m < n - N$. However, by definition of $N$, this means that there exists $i \in [m - N, m + N] \subseteq [0, n]$ so that $d(T^i x, T^i y) > \delta$, completing the proof.

\end{proof}

The following fact actually appears in \cite{bowen} (as Lemma 1), but as it is a simple corollary of Lemma~\ref{sepscaleindep}, we give a proof here as well.

\begin{corollary}\label{partscaleindep}
If $(X,T)$ is expansive (with expansivity constant $\delta$) and $\eta < \delta$, then for any potential $\phi$,
\[
Z(X, T, \phi, n, \eta) \leq M(\eta) Z(X, T, \phi, n, \delta),
\]
where $M(\eta)$ is as in Lemma~\ref{sepscaleindep}.
\end{corollary}

\begin{proof}
Consider an $(n,\eta)$-separated set $U$ for which $\sum_{x \in S} e^{S_n \phi(x)} = Z(X,T, \phi,n,\eta)$. Then by Lemma~\ref{sepscaleindep}, we can write $U = \bigcup_{i = 1}^{M(\eta)} U_i$, where each $U_i$ is $(n, \delta)$-separated. Then
\[
Z(X,T, \phi,n,\eta) = \sum_{x \in U} e^{S_n \phi(x)} = \sum_{i = 1}^{M(\eta)} \left(\sum_{x \in U_i} e^{S_n \phi(x)} \right) \leq M(\eta) Z(X, T, \phi, n, \delta).
\]
\end{proof}

\section{Proofs of Theorems~\ref{mainthm} and \ref{mainthm2}}\label{proofs}

The first tool that we need is quite basic; it is the existence of a sequence with helpful bounds on $f$ and $g$ under the hypotheses of either Theorem~\ref{mainthm} or \ref{mainthm2}.

\begin{lemma}\label{goodseq}
If $(X,T)$ and $\phi$ satisfy the hypotheses of either Theorem~\ref{mainthm} or \ref{mainthm2}, then there exists a sequence $\{n_k\}$ so that for all $\eta < \delta$, there exist increasing 
$\overline{f}$ and $\overline{g}$ where $(X,T)$ has non-uniform transitivity with gap bounds 
$\overline{f}(n)$ at scale $\eta$, $\phi$ has partial sum variation bounds $\overline{g}(n)$ at scale 
$\eta$, and $\overline{f}(n_k)/\ln n_k, \overline{g}(n_k)/\ln n_k \rightarrow 0$. 
\end{lemma}

\begin{proof}

We first note that by Lemmas~\ref{gapscaleindep} and \ref{sumscaleindep}, there are positive constants $C,D$ so that we can take $\overline{f}(n) = C + f(n + C)$ and $\overline{g}(n) = D + g(n) \leq D + g(n + C)$. 

Under the hypotheses of Theorem~\ref{mainthm}, there exists a sequence $\{m_k\}$ so that 
$\frac{f(m_k) + g(m_k)}{\ln m_k} \rightarrow 0$. Then if we take $n_k = m_k - C$, $\frac{\overline{f}(n_k) + \overline{g}(n_k)}{\ln n_k} \rightarrow 0$, so $\{n_k\}$ satisfies the conclusion of the lemma. 

Similarly, under the hypotheses of Theorem~\ref{mainthm2}, $\frac{f(k) + g(k)}{\ln k} \rightarrow 0$, so $\frac{\overline{f}(k) + \overline{g}(k)}{\ln k} \rightarrow 0$ also, so taking $n_k = k$ satisfies the conclusion of the lemma. 
\end{proof}

\begin{definition}
A sequence $\{n_k\}$ satisfying the conclusion of Lemma~\ref{goodseq} is called an \textbf{anchor sequence} for $(X,T)$ and $\phi$.
\end{definition}

The next tools that we'll need are some upper bounds on the partition function for expansive systems with non-uniform specification, which generalize the well-known upper bound of of $e^{nP(X,T,\phi)}$ times a constant when specification is assumed (e.g. Lemma 3 of \cite{bowen}). 

\begin{theorem}\label{specpartbdthm}
If $(X,T)$ is an expansive dynamical system (with expansivity constant $\delta$) and non-uniform specification with gap bounds 
$f(n)$ at scale $\delta/3$, and $\phi$ is a potential with partial sum variation bounds $g(n)$ at scale $\delta/3$, then for all $n$, 
\[
Z(X, T, \phi, n,\delta) \leq e^{(n + f(n)) P(X,T, \phi) - f(n) \inf \phi + g(n)}.
\]
\end{theorem}

\begin{proof}
Suppose that $(X,T)$ and $\phi$ are as in the theorem, denote $m = \inf \phi$, and fix any $n$. Then, choose any $k \in \mathbb{N}$ and an $(n,\delta)$-separated set $S$ so that $\sum_{x \in S} e^{S_n \phi(x)} = Z(X,T, \phi,n,\delta)$. Then, for any $x_1, \ldots, x_k \in S$, we can use non-uniform specification to choose a point $y = y(x_1, \ldots, x_k)$ which $\delta/3$-shadows $(x_1, \ldots, x_k)$ for $(n,\ldots,n)$ iterates, with gaps $(f(n), \ldots, f(n))$. Then
\[
S_{k(n+f(n))}\phi(y) \geq kmf(n) + \sum_{j=1}^{k} S_n \phi(T^{(j-1)(n + f(n))} y) \geq kmf(n) - kg(n) + \sum_{j=1}^{k} S_n \phi(x_j).
\]

We note that the set $Y = \{y(x_1, \ldots, x_k) \ : \ x_i \in S\}$ is $(k(n+f(n)),\delta/3)$-separated by definition, and so 
\begin{multline*}
Z(X, T,\phi, k(n+f(n)), \delta/3) \geq \sum_{y \in Y} e^{S_{k(n+f(n))}\phi(y)} \\
\geq e^{kmf(n) - kg(n)} \sum_{x_1, \ldots, x_k \in S} \prod_{i=1}^k e^{S_n \phi(x_i)} = e^{kmf(n) - kg(n)} Z(X, T,\phi, n, \delta)^k.
\end{multline*}

Taking logarithms, dividing by $k(n+f(n))$, and letting $k \rightarrow \infty$ yields
\[
P(X,T, \phi) = P(X,T,\phi,\delta/3) \geq \frac{\ln Z(X,T, \phi, n, \delta) + mf(n) - g(n)}{n + f(n)}.
\]
(The first equality comes from Lemma~\ref{presscaleindep}.) Now, solving for $Z(X,T,\phi,n,\delta)$ completes the proof.
%However, by (\ref{subadd}),
%\[
%|\mathcal{L}_{k(n + f(n)}(X)| \leq e^{k(n + f(n) h(X)}.
%\]
%Combining these inequalities completes the proof.

\end{proof}

\begin{corollary}\label{speccor}
If $(X,T)$ satisfies the hypotheses of Theorem~\ref{mainthm} and $\{n_k\}$ is an anchor sequence, then for every 
$\epsilon > 0$, there exists $K$ so that for all $k > K$ and $1 \leq i \leq n_k$, 
\[
Z(X,T, \phi, i, \delta) \leq e^{iP(X,T, \phi)} n_k^{\epsilon}.
\]
\end{corollary}

\begin{proof}
By definition of anchor sequence, we can take $f(n)$ and $g(n)$ to be gap bounds and partial sum variation bounds at scale $\delta/3$, and can choose $K$ so that for $k > K$, $f(n_k) + g(n_k) < \frac{\epsilon}{P(X,T, \phi) + |\inf \phi| + 1} \ln n_k$. Then, for any such $k$ and $1 \leq i \leq n_k$, Theorem~\ref{specpartbdthm} (along with monotonicity of $f,g$) implies  
\begin{multline*}
Z(X,T, \phi, i, \delta) \leq e^{(i + f(i)) P(X,T,\phi) - f(i) \inf \phi + g(i)} \leq \\
e^{iP(X,T, \phi)} e^{(P(X,T, \phi) + |\inf \phi| + 1)(f(i) + g(i))} \leq e^{iP(X,T, \phi)} n_k^{\epsilon}.
\end{multline*}

\end{proof}

We now prove a somewhat similar bound under the assumption of non-uniform transitivity, which requires information about $f,g$ for all large $n$ rather than a single value.

\begin{theorem}\label{transpartbdthm}
If $(X,T)$ is an expansive dynamical system (with expansivity constant $\delta$) and non-uniform transitivity with gap bounds $f(n)$ at scale $\delta/3$, $\phi$ is a potential with partial sum variation bounds $g(n)$, and $C > 0, M \geq 3$ satisfy $f(n) + g(n) \leq \min(C\ln n,n)$ for all $n \geq M$, then for all $n \geq M$, 
\[
Z(X,T, \phi, n, \delta) \leq D e^{nP(X,T, \phi)} n^{CE},
\]
where $D,E$ are constants depending only on $X$ and $\phi$.
\end{theorem}

\begin{proof}
Suppose that $X$, $T$, $\phi$, $C$, and $M$ are as in the theorem, and fix any $n \geq M$. We use $m$ to denote $\inf \phi$. For every $j$, choose $U_j$ a $(j, \delta)$-separated set for which $\sum_{x \in U_j} e^{S_j \phi(x)} = Z(X, T, \phi, j, \delta)$. We will use non-uniform transitivity to give lower bounds on $Z(X, T, \phi, n_k, \delta)$ for a recursively defined sequence $\{n_k\}$. Define $n_0 = n$, and for $k \geq 0$ define $n_{k+1} = 2n_k + f(n_k)$. Note that all $n_k \geq M$, and so $n_{k+1} \leq 3n_k$ for all $k$, meaning that $n_k \leq 3^k n$ for all $k$. For a better bound, we see that by induction, for every $k$,
\begin{multline}\label{nkbd}
n_k = 2^k n + 2^{k-1} f(n_1) + 2^{k-2} f(n_2) + \ldots + f(n_{k-1}) \\
\leq 2^k n + 2^k C [2^{-1} \ln (3n) + 2^{-2} \ln (3^2 n) + \ldots + 2^{-(k-1)} \ln (3^k n)] \\
= 2^k (n + C\ln (9n)).
\end{multline}

Then, for every $k \geq 0$, and for any $x, y \in U_{n_k}$, use non-uniform transitivity to create a point $z(x,y)$ which $\delta/3$-shadows $(x,y)$ for $(n_k,n_k)$ iterates, with a gap $i$ of length less than or equal to $f(n_k)$. Then,
\begin{multline*}
S_{n_{k+1}} \phi (z(x,y)) \geq S_{n_k} \phi(z(x,y)) + S_{n_k} \phi(T^{i + n_k} z(x,y)) - |m|f(n_k) \\
\geq S_{n_k} \phi(x) + S_{n_k} \phi(y) - 2g(n_k) - |m|f(n_k).
\end{multline*}

This implies that
\begin{multline*}
\sum_{z(x,y)} e^{S_{n_{k+1}} \phi(z(x,y))} \geq \sum_{x,y} e^{S_{n_k} \phi(x) + S_{n_k} \phi(y) - 2g(n_k) - |m|f(n_k)} \\
= e^{-2g(n_k) - |m|f(n_k)} \left(\sum_{x \in U_{n_k}} e^{S_{n_k} \phi(x)} \right)^2 = e^{-2g(n_k) - |m|f(n_k)} Z(X, T, \phi, n_k, \delta)^2.
\end{multline*}

Then there is a single gap $i_k \leq f(n_k)$ such that if we define $Y_k$ to be the set of $(x,y)$ for which $z(x,y)$ used a gap of length $i$, then
\begin{multline}\label{transbd1}
\sum_{(x,y) \in Y_k} e^{S_{n_{k+1}} \phi(z(x,y))} \geq (f(n_k))^{-1} e^{-2g(n_k) - |m|f(n_k)} Z(X, T, \phi, n_k, \delta)^2 \\
\geq e^{-2g(n_k) - (|m|+1)f(n_k)} Z(X, T, \phi, n_k, \delta)^2.
\end{multline}

We now claim that $\{z(x,y) \ : \ (x,y) \in Y_k\}$ is $(n_{k+1}, \delta/3)$-separated. To see this, choose any unequal 
$(x,y), (x', y') \in Y_{n_{k+1}}$, and write $z = z(x,y)$ and $z' = z(x', y')$. Either $x \neq x' \in U_{n_k}$ or $y \neq y' \in U_{n_k}$. In the former case, since $U_{n_k}$ is $(n,\delta)$-separated, there exists $0 \leq j < n_k$ so that $d(T^j x, T^j x') > \delta$, and by definition of $z$ and $z'$, $d(T^j z, T^j z') > \delta/3$. In the latter case, there exists $0 \leq j < n_k$ so that $d(T^j y, T^j y') > \delta$, and again by definition of $z$ and $z'$, $d(T^{n_k + i_k + j} z, T^{n_k + i_k + j} z) > \delta/3$. So, $\{z(x,y) \ : \ (x,y) \in Y_k\}$
is $(n_{k+1}, \delta/3)$-separated as claimed. Therefore, Lemma~\ref{partscaleindep} and (\ref{transbd1}) imply
\begin{multline*}%\label{transpartbd}
Z(X, T, \phi, n_{k+1}, \delta) \geq M(\delta/3)^{-1} Z(X, T, \phi, n_{k+1}, \delta/3) \geq \sum_{(x,y) \in Y_k} e^{S_{n_{k+1}} \phi(z(x,y))}\\
\geq M(\delta/3)^{-1} e^{-2g(n_k) - (|m|+1)f(n_k)} Z(X, T, \phi, n_k, \delta)^2.
\end{multline*}

Now, by induction, $\ln Z(X, T, \phi, n_k, \delta)$ is greater than or equal to
\begin{multline}\label{transbd2}
- k \ln M(\delta/3) + 2^k \ln Z(X, T, \phi, n, \delta) - \sum_{i = 0}^{k-1} 2^{k-i-1} (2g(n_i) - (m+1)f(n_i))\\
\geq - k \ln M(\delta/3) + 2^k \ln Z(X, T, \phi, n, \delta) - 2^k \sum_{i = 0}^{k-1} 2^{-i-1} (2 + |m+1|) C \ln (3^i n)\\
= -k \ln M(\delta/3) + 2^k(\ln Z(X, T, \phi, n, \delta) - C (2 + |m+1|) \ln(9n)).  
\end{multline}

Therefore, by (\ref{nkbd}) and (\ref{transbd2}),
\[
\frac{\ln Z(X, T, \phi, n_k, \delta)}{n_k} \geq \frac{-k \ln M(\delta/3) + 2^k(\ln Z(X, T, \phi, n, \delta) - C (2 + |m+1|) \ln(9n))}{2^k(n + C\ln(9n))}.
\]

Letting $n_k \rightarrow \infty$ yields
\[
P(X,T, \phi) \geq \frac{\ln Z(X, T, \phi, n, \delta) - C (2 + |m+1|) \ln(9n)}{n + C\ln(9n)},
\]
and we can rewrite as
\[
Z(X, T, \phi, n, \delta) \leq D e^{Pn} n^{CE},
\]
(here $D = 9^{C(P(X,T, \phi) + 2 + |m-1|)}$ and $E = P(X,T, \phi) + 2 + |m-1|$), completing the proof.

\end{proof}

\begin{corollary}\label{transcor}
If $(X,T)$ and $\phi$ satisfy the hypotheses of Theorem~\ref{mainthm2} and $\{n_k\}$ is an anchor sequence, then for every $\epsilon > 0$, there exists $K$ so that for any $k > K$ and $1 \leq i \leq n_k$,
\[
Z(X, T, \phi, i, \delta) \leq e^{iP(X,T, \phi)} n_k^{\epsilon}.
\]
\end{corollary}

\begin{proof}
Under the hypotheses of Theorem~\ref{mainthm2}, we may choose $M \geq 3$ so that for $n \geq M$, $f(n) + g(n) < \min((\epsilon/2E) \ln n,n)$, where $E$ is as in Theorem~\ref{mainthm2}. Then, for all $n \geq M$,
\[
Z(X, T, \phi, i, \delta) \leq D e^{iP(X,T, \phi)} i^{\epsilon/2}.
\]
This can clearly be improved to hold for all $i$ by changing $D$, i.e. there exists $D'$ so that for all $i$,
\[
Z(X, T, \phi, i, \delta) \leq D' e^{iP(X,T, \phi)} i^{\epsilon/2}.
\]
We then just choose $K$ so that $n_k^{\epsilon/2} > D'$. Then, for any $k \geq K$ and $1 \leq i \leq n_k$,
\[
Z(X, T, \phi, i, \delta) \leq D' e^{iP(X,T, \phi)} i^{\epsilon/2} < e^{iP(X,T, \phi)} n_k^{\epsilon},
\]
completing the proof.
\end{proof}

We will also need the following general lower bound on the sum of $e^{S_n \phi (x)}$ over separated sets within a set of positive measure for an equilibrium state. 

\begin{theorem}\label{measbd}
If $(X,T)$ is an expansive dynamical system (with expansivity constant $\delta$), $\phi$ is a potential with partial sum variations $g(n)$ at scale $\delta/3$, and $A \subset X$ has $\mu(A) > 0$ for some equilibrium state $\mu$ of $X$ for $\phi$, then there exists an $(n,\delta/3)$-separated subset $U$ of $X$ with
\[
\sum_{x \in U} e^{S_n \phi(x)} \geq \left(e^{nP(X,T,\phi)}\right)^{1/\mu(A)} \left(Z(X,T, \phi,n,\delta/3)\right)^{(\mu(A)-1)/\mu(A)} M^{-1} e^{-g(n) - \frac{\ln 2}{\mu(A)}},
\]
where $M = M(\delta/3)$ as defined in Lemma~\ref{sepscaleindep}.
\end{theorem}

\begin{proof}
Consider such $X$, $T$, $\phi$, $\mu$, $A$, and $n$. Choose a maximal $(n, \delta/3)$-separated subset $U$ of $X$. As in \cite{bowen} or \cite{CT2}, we can create a partition $\mathcal{P} = \{A_x\}_{x \in U}$ where for each $x \in U$ and $y \in A_x$, $y$ $\delta/3$-shadows $x$ for $n$ iterates, i.e. $d(T^i x, T^i y) < \delta/3$ for $0 \leq i < n$. 

Then, if two points are in the same element of $\bigvee_{m \in \mathbb{Z}} T^{mn} \mathcal{P}$, by expansivity they are equal, i.e. $\mathcal{P}$ is a generating partition for $(X,T^n,\mu)$. This means that 
\begin{multline}\label{eqn1}
nP(X, T, \phi) = n\left(h(X, T, \mu) + \int \phi \ d\mu\right) 
= h(X, T^n, \mu) + \int S_n \phi \ d\mu \\ = h(X, T^n, \mu, \mathcal{P}) + \int S_n \phi \ d\mu 
\leq H(X,T^n,\mu,\mathcal{P}) + \int S_n \phi \ d\mu \\
\leq \sum_{x \in U} \mu(A_x) \left[-\ln \mu(A_x) + \sup_{y \in A_x} S_n \phi(y)\right] \leq g(n) + \sum_{x \in U} \mu(A_x) \left[- \ln \mu(A_x) + S_n \phi(x)\right].
\end{multline}

We write $U' = \{x \in U \ : \ A_x \cap A \neq \varnothing\}$, and then can break up the final sum:
\begin{multline}\label{eqn2}
\sum_{x \in U} \mu(A_x) \left[-\ln \mu(A_x) + S_n \phi(x)\right] = \sum_{x \in U'} \mu(A_x) \left[-\ln \mu(A_x) + S_n \phi(x)\right] \\ + \sum_{x \in U \setminus U'} \mu(A_x) \left[-\ln \mu(A_x) + S_n \phi(x)\right].
\end{multline}

For fixed positive reals $a_1, \ldots, a_n$ and positive $p_1, \ldots, p_N$ with fixed sum $S$, 
$\sum p_i(a_i - \ln p_i)$ has maximum value $S (-\ln S + \ln\sum e^{a_i})$. (See Lemma 9.9 in \cite{walters}.)
Therefore, if we write $A' = \bigcup_{x \in U'} A_x$, then 
\begin{multline}\label{eqn3}
\sum_{x \in U'} \mu(A_x) \left[-\ln \mu(A_x) + S_n \phi(x)\right] \leq 
\mu(A') \left(-\ln \mu(A') + \ln \sum_{x \in U'} e^{S_n \phi(x)}\right).
\end{multline}

Similarly,
\begin{multline}\label{eqn4}
\sum_{x \in U \setminus U'} \mu(A_x) \left[-\ln \mu(A_x) + S_n \phi(x)\right] \\
\leq (1 - \mu(A')) \left(-\ln (1 - \mu(A')) + \ln \sum_{x \in U \setminus U'} e^{S_n \phi(x)}\right)\\
\leq (1 - \mu(A')) \left(\ln (1 - \mu(A')) + \ln Z(X, T, n, \phi, \delta/3)\right).
\end{multline}

Combining (\ref{eqn1})-(\ref{eqn4}) yields

\begin{multline*}
nP(X, T,\phi) \leq \mu(A') \left(-\ln \mu(A') + \ln \sum_{x \in U'} e^{S_n \phi(x)}\right) + \\
(1 - \mu(A')) \left(\ln (1 - \mu(A')) + \ln Z(X, T, \phi, n, \delta/3)\right)\\ \leq \ln 2 + \mu(A') \ln \sum_{x \in U'} e^{S_n \phi(x)} + (1 - \mu(A')) \ln Z(X, T, \phi,n, \delta/3).
\end{multline*} 

Finally, we solve for $\ln \sum_{x \in U'} e^{S_n \phi(x)}$:
\begin{multline}\label{bigeqn}
\ln \sum_{x \in U'} e^{S_n \phi(x)} \geq \frac{nP(X, T,\phi)}{\mu(A')} - \frac{1 - \mu(A')}{\mu(A')} \ln Z(X, T, \phi, n, \delta/3) - \frac{\ln 2}{\mu(A')}\\
%\\ = nP(X, \phi) - \frac{1 - \mu(A')}{\mu(A')} (\ln Z(X, n, \phi, \delta/3) - nP(X, \phi)) - \frac{\ln 2}{\mu(A')}\\
\geq \frac{nP(X,T, \phi)}{\mu(A)} - \frac{1 - \mu(A)}{\mu(A)} \ln Z(X, T, \phi, n, \delta/3) - \frac{\ln 2}{\mu(A)}.
\end{multline}
(Recall here that $\mu(A') \geq \mu(A)$.)

Now, by Lemma~\ref{sepscaleindep}, we can write $U'$ as the union of $M = M(\delta/3)$ sets which are each $(n, \delta)$-separated. Then (\ref{bigeqn}) implies that there must exist one, call it $U''$, so that
\begin{equation}\label{bigeqn3}
\sum_{x \in U''} e^{S_n \phi(x)} \geq M^{-1} \left(e^{nP(X,T, \phi)}\right)^{1/\mu(A)} \left(Z(X,T,\phi,n,\delta/3)\right)^{(\mu(A) - 1)/\mu(A)} e^{-\frac{\ln 2}{\mu(A)}}.
\end{equation}

Recall that for every $x \in U''$, $A_x \cap A \neq \varnothing$. Therefore, we can define a new set $U'''$ which contains a single point from each $A_x \cap A$ for $x \in U''$. Recall that $U''$ was $(n, \delta)$-separated, and that every point of $A_x$ $\delta/3$-shadows $x$ for $n$ iterates; therefore, $U'''$ is $(n, \delta/3)$-separated. Then by (\ref{bigeqn3}), 
\begin{multline*}%\label{bigeqn4}
\sum_{x \in U'''} e^{S_n \phi(x)} \geq \sum_{x \in U''} e^{S_n \phi(x) - g(n)}\\
\geq M^{-1} \left(e^{nP(X,T, \phi)}\right)^{1/\mu(A)} \left(Z(X,T,\phi,n,\delta/3)\right)^{(\mu(A) - 1)/\mu(A)} e^{-g(n) - \frac{\ln 2}{\mu(A)}},
\end{multline*}
completing the proof.

\end{proof}

The following is an immediate corollary of Theorem~\ref{measbd} and Corollaries~\ref{partscaleindep}, \ref{speccor}, and \ref{transcor}. 

\begin{corollary}\label{maincor}
If $(X,T)$ and $\phi$ satisfy the hypotheses of either Theorem~\ref{mainthm} or \ref{mainthm2} and $\{n_k\}$ is an anchor sequence, then for every $\epsilon > 0$, there exists $K$ so that for any $A \subset X$ with $\mu(A) > 1/2$ for some equilibrium state $\mu$ of $X$ for $\phi$, any $k \geq K$, and any $1 \leq i \leq n_k$, there is an $(i,\delta/3)$-separated subset $T$ of $A$ with
\[
\sum_{x \in T} e^{S_i \phi(x)} \geq n^{-\epsilon} e^{iP(X,T, \phi)}.
\]
\end{corollary}

We may now prove our main results. 

\begin{proof}[Proof of Theorems~\ref{mainthm} and Theorem~\ref{mainthm2}]
Choose $X$ and $\phi$ satisfying the hypotheses of either Theorem~\ref{mainthm} or \ref{mainthm2}, and a corresponding anchor sequence 
$\{n_k\}$. Define $m = \inf \phi$. Suppose for a contradiction that $X$ has more than one equilibrium state for $\phi$. Then, as noted in the introduction, $X$ has ergodic equilibrium states $\mu \neq \nu$ for $\phi$, and $\mu \perp \nu$. Then, there exists measurable $R \subset X$ with $\mu(R) = \nu(R^c) = 1$. Since $\mu, \nu$ are Borel measures, there exist open sets $U \supseteq R$ and $U' \supseteq R^c$ so that $\mu(U), \nu(U') < \frac{1}{5}$. 
We write $\eta = \min(\delta/9, d(U^c, U'^c)/3) > 0$.

Define $W \subset X$ to be the set of $x \in X$ for which
\[
M^+(\chi_{U}) = \sup_N \frac{1}{N} \sum_{n=0}^{N-1} \chi_U(T^n x) \leq 2\mu(U) < \frac{2}{5}.
\]
In other words, for every point of $W$ and any $i$, fewer than $\frac{2}{5}$ of its first $i$ iterates under $T$ are in $U$. 
By Corollary~\ref{maxergcor} (to the Maximal Ergodic Theorem), $\displaystyle \mu(W) \geq 1 - \frac{\int \chi_U \ d\mu}{2\mu(U)} = \frac{1}{2}$. 

Similarly, we define $V \subset X$ to be the set of $x \in X$ for which 
\[
M^- (\chi_{T^m U'}) = \sup_N \frac{1}{N} \sum_{n=0}^{N-1} \chi_{T^m U'}(T^{-n} x) \leq 2\nu(U') < \frac{2}{5}.
\]
For every point of $V$ and any $i$, fewer than $\frac{2}{5}$ of its first $i$ iterates under 
$T^{-1}$ are in $U'$. 
By Corollary~\ref{maxergcor} (applied to $T^{-1}$), $\displaystyle \mu(V) \geq 1 - \frac{\int \chi_{T^m U'} \ d\nu}{2\nu(U')} = \frac{1}{2}$. 

By Corollary~\ref{maincor}, there exists $K$ so that for $k \geq K$ and all $j \leq n_k$, we can define $(j, \delta/3)$-separated sets 
$V_j \subset V$ and $W_j \subset W$ for which 
 
\begin{equation}\label{VWbd}
\sum_{x \in V_j} e^{S_j\phi(x)}, \sum_{x \in W_j} e^{S_j\phi(x)} \geq e^{jP(X,T, \phi)} n_k^{-1/5}.
\end{equation}

Whether $(X,T)$ and $\phi$ satisfied the hypotheses of Theorem~\ref{mainthm} or \ref{mainthm2}, by Lemmas~\ref{gapscaleindep} and \ref{sumscaleindep} we may assume that $(X,T)$ satisfies non-uniform transitivity at scale $\eta$ with gap bounds $f(n)$ and that $\phi$ has partial sum variation bounds $g(n)$ at scale $\eta$. The final step is to use non-uniform transitivity of $X$ to create an $(n_k,\eta)$-separated set by shadowing orbit segments from various $V_i$ and $W_j$. For large $k$, the sum of $e^{S_{n_k} \phi(x)}$ over this set will be large enough to contradict Corollary~\ref{speccor} or Corollary~\ref{transcor}.

By definition of anchor sequence, we can increase $K$ so that for any $k \geq K$, $f(n_k), g(n_k) < \frac{1}{5(P(X,T, \phi) + |m| + 2)} \ln n_k$. Choose any $k \geq K$ and define $n := n_k$ for ease of notation. Then, for any integer $j$ in $[1, \frac{n - f(n)}{2f(n)}]$, any $v \in V_{2jf(n)}$, and any $w \in W_{n - f(n) - 2jf(n)}$, we use non-uniform transitivity to create a point $x(j,v,w)$ which $\eta$-shadows $(v,w)$ for $(2j f(n), n - f(n) - 2j f(n))$ iterates with gap $i \leq f(n)$. We first note that for all $j,v,w$,
\begin{multline*}%\label{partbd2}
S_n \phi(x(j,v,w)) \geq \\
S_{2jf(n)} \phi(v) + S_{n - (2j+1)f(n)} \phi(w) - |m| f(n) - g(2jf(n)) - g(n - (2j+1)f(n)) \\ 
\geq S_{2jf(n)} \phi(v) + S_{n - (2j+1)f(n)} \phi(w) - |m| f(n) - 2g(n)\\
\geq S_{2jf(n)} \phi(v) + S_{n - 2jf(n) - f(n)} \phi(w) - 1/5 \ln n.
\end{multline*}

This implies that for any $j$,
\begin{multline*}
\sum_{v,w} e^{S_n \phi (x(j,v,w))} \geq n^{-1/5} \left(\sum_v e^{S_{2jf(n)} \phi(v)}\right) 
\left(\sum_w e^{S_{n - 2jf(n) - f(n)} \phi(w)}\right)\\ 
\stackrel{(\ref{VWbd})}{\geq} e^{P(X,T,\phi)(n - f(n))} n^{-3/5} \geq e^{nP(X,T,\phi)} n^{-4/5}.
\end{multline*}

Then there exists a set $T(j)$ of pairs $(v,w)$ so that all use the same gap $i_j \leq f(n)$, and
\begin{equation}\label{partbd2.5}
\sum_{v,w \in T(j)} e^{S_n \phi(x(j,v,w))} \geq (f(n))^{-1} e^{nP(X,T,\phi)} n^{-4/5}.
\end{equation}

We claim that the set $Z = \bigcup_{j=1}^{\lfloor (n-f(n))/2f(n) \rfloor} \{x(j,v,w) \ : \ (v,w) \in T(j)\}$ is $(n,\eta)$-separated. To see this, choose any triples $(j,v,w) \neq (j',v',w')$ with $(v,w) \in T(j)$ and $(v',w') \in T(j')$. We break into the cases $j = j'$ and $j \neq j'$, and for brevity write $x = x(j,v,w)$ and $x' = x(j',v',w')$.

If $j = j'$, then either $v \neq v'$ or $w \neq w'$, and since $V_i$ and $W_i$ are $(i, \delta/3)$-separated for all $i$, either $d(T^k v, T^k v') > \delta/3$ for some $0 \leq k < 2j f(n)$ or $d(T^k w, T^k w') > \delta/3$ for some $0 \leq k < n - f(n) - 2j f(n)$. In the first case, since $x$ $\eta$-shadows $v$ for its first $2jf(n)$ iterates and $x'$ $\eta$-shadows $v'$ for its first $2j f(n)$ iterates (recall that $j = j'$), $d(T^k x, T^k x') > \delta/3 - 2\eta \geq \eta$. The second case is trivially similar, using the $\eta$-shadowing of $w$ and $w'$ and the fact that since $(v,w), (v', w') \in T(j)$, the gaps used for $x$ and $x'$ are equal.

Now suppose that $j \neq j'$, and without loss of generality assume $j < j'$. Recall that $i_j$ denotes the gap used in the construction of $x$. Then by definition of $x$, $T^{2j f(n) + i_j} x$ $\eta$-shadows $w$ for $n - 2j f(n) - f(n) \geq (2j' - 2j) f(n) - i_j + 1$ iterates. Since $w \in W_{n - f(n) - 2jf(n)} \subset W$, fewer than $\frac{2}{5} ((2j' - 2j) f(n) - i_j + 1)$ of the first $(2j' - 2j) f(n) - i_j + 1$ iterates of $w$ are in $U$. Therefore, more than a proportion of $\frac{3}{5}$ of the points $\{T^m x \ : \ 2j f(n) + i_j \leq m \leq 2j'f(n)\}$ are within distance $\eta$ of $U^c$.

Similarly, by definition, $x'$ $\eta$-shadows $v'$ for $2j' f(n) \geq (2j' - 2j) f(n) - i_j + 1$ iterates. Since $v' \in V_{2j'f(n)}$, $T^{2j'f(n)} v' \in V$, and so the proportion of the first 
$(2j' - 2j) f(n) - i_j + 1$ iterates under $T^{-1}$ of $T^{2j'f(n)} v'$ which are in $U'$ is less than $\frac{2}{5}$. Therefore, more than a proportion of $\frac{3}{5}$ of the points $\{T^m x' \ : \ 2j f(n) + i_j \leq m \leq 2j'f(n)\}$ are within distance $\eta$ of $U'^c$. So, there exists $k \in [2j f(n) + i_j, 2j'f(n)]$ so that $d(T^k x, U^c), d(T^k x', U'^c) < \eta$. By definition of $\eta$, $d(U^c, U'^c) \geq 3\eta$. Therefore, $d(T^k x, T^k x') > \eta$. 

In either case, we've shown that there exists $0 \leq k < n$ for which $d(T^k x, T^k x') > \eta$, and so $Z$ is $(n, \eta)$-separated as claimed. Then,
\begin{multline*}
\sum_{x \in Z} e^{S_n \phi(x)} = \sum_{j = 1}^{\lfloor (n-f(n))/2f(n) \rfloor} \sum_{(v,w) \in T(j)} e^{S_n \phi(x(j,v,w))}\\
\stackrel{(\ref{partbd2.5})}{\geq} \frac{n-f(n)}{2f(n)^2} e^{nP(X,T,\phi)} n^{-4/5}.
\end{multline*}

As before, we can use Lemma~\ref{sepscaleindep} to pass to $T' \subseteq T$ which is $(n, \delta)$-separated and for which 
\[
\sum_{x \in T'} e^{S_n \phi(x)} \geq \frac{n - f(n)}{2f(n)^2} (M(\eta))^{-1} e^{n P(X,T, \phi)} n^{-4/5},
\]
implying that
\[%\label{finalbd}
Z(X,T,\phi,n,\delta) \geq \frac{n - f(n)}{2f(n)^2} (M(\eta))^{-1} e^{n P(X,T, \phi)} n^{-4/5}.
\]

However, this will contradict Corollary~\ref{speccor} or \ref{transcor} for large $k$. Therefore, our original assumption of multiple equilibrium states on $X$ was false, and $X$ has a unique equilibrium state, which we denote by $\mu$.\\

It remains to show that $\mu$ is fully supported, and so for a contradiction assume that there is a nonempty open set $U \subset X$ with $\mu(U) = 0$. Then the set $Y$ of points whose orbits under $T$ never visit $U$ has $\mu(Y) = 1$, and $Y$ contains some open ball $B_{\rho}(y)$. Define 
$\eta = \min(\delta/9, \rho/3)$. By the definition of anchor sequence, there exists $K$ so that for $k \geq K$, $f(n_k), g(n_k) < \frac{1}{3(3P(X, T, \phi) + 3|m| + 2)} \ln n_k$.

By Theorem~\ref{measbd}, for every $n$ there exists an $(n, \delta/3)$-separated set $Y_n \subseteq Y$ with 
\begin{equation}\label{Ybd}
\sum_{x \in Y_n} e^{S_n \phi(x)} \geq C e^{nP(X,T, \phi) - g(n)},
\end{equation}
where $C = (2M(\delta/3))^{-1}$.

Now, we will again use non-uniform transitivity to obtain a contradiction to one of Corollary~\ref{speccor} or Corollary~\ref{transcor}. Choose $k \geq K$ and denote $n := n_k$. Then, for any integer $j$ in $[1, \frac{n - 2f(n)}{2f(n)}]$, any $v \in U_{2jf(n)}$, and any $w \in U_{n - (2j+2)f(n) - 1}$, we use non-uniform transitivity to choose $z(j,v,w) \in X$ which $\eta$-shadows $(v,y,w)$ for $(2jf(n),1,n - (2j+2)f(n) - 1)$ iterates, with gaps $(i,i')$ both less than or equal to $f(n)$. We first note that for all $j,v,w$,
\begin{multline*}%\label{partbd3}
S_n \phi(z(j,v,w)) \geq \\
S_{2jf(n)} \phi(v) + S_{n - (2j+2)f(n) - 1} \phi(w) - |m|(2f(n)+1) - g(2jf(n)) - g(n - (2j+2)f(n) - 1) \\ 
\geq S_{2jf(n)} \phi(v) + S_{n - (2j+2)f(n) - 1} \phi(w) - |m|(2f(n) + 1) - 2g(n)\\
\geq S_{2jf(n)} \phi(v) + S_{n - (2j+2) f(n) - 1} \phi(w) - 1/3 \ln n.
\end{multline*}

Then, for any $j$,
\begin{multline*}
\sum_{v,w} e^{S_n \phi(z(j,v,w))} \geq n^{-1/3} \left( \sum_{v} e^{S_{2jf(n)} \phi(v)} \right) 
\left( \sum_{w} e^{S_{n - (2j+2)f(n) - 1} \phi(w)} \right)\\
\stackrel{(\ref{Ybd})}{\geq} C^2 e^{(n - 2f(n) - 1)P(X,T,\phi) - g(2jf(n)) - g(n - (2j + 2) f(n) - 1)} n^{-1/3}\\
\geq C^2 e^{(n - 2f(n) - 1)P(X,T,\phi) - 2g(n)} n^{-1/3} \geq C^2 e^{nP(X,T,\phi)} n^{-2/3}. 
\end{multline*}

Then there exists a set $U(j)$ of pairs $(v,w)$ so that all use the same gaps $i_j, i'_j \leq f(n)$, and \begin{equation}\label{partbd3.5}
\sum_{v,w \in U(j)} e^{S_n \phi(z(j,v,w))} \geq (f(n))^{-2} C^2 e^{nP(X,T,\phi)} n^{-2/3}.
\end{equation}

We claim that the set $Z = \bigcup_{j=1}^{\lfloor (n-f(n))/2f(n) \rfloor} \{x(j,v,w) \ : \ (v,w) \in U(j)\}$ is $(n,\eta)$-separated. To see this, choose any triples $(j,v,w) \neq (j',v',w')$ with $(v,w) \in U(j)$ and $(v',w') \in U(j')$. We break into the cases $j = j'$ and $j \neq j'$, and for brevity write $z = z(j,v,w)$ and $z' = z(j',v',w')$. If $j = j'$, then the proof that there exists $0 \leq k < n$ for which $d(T^k z, T^k z') > \eta$ is the same as was done above in the proof of uniqueness of $\mu$ (again, recall that $z$ and $z'$ both use the same gap $j$.)

If $j \neq j'$, then without loss of generality we assume $j < j'$, and recall that $i_j$ denotes the first gap used for $z$. Then by definition of $z$, 
$d(T^{2jf(n) + i_j + 1} z, y) < \eta$.
Similarly, since $2jf(n) + i_j + 1 < 2j'f(n)$, $d(T^{2jf(n) + i_j + 1} z', T^{2j f(n) + i_j + 1} v) < \eta$ by definition of $z'$.
Recall that $v \in Y$, and so $T^{2j f(n) + i_j + 1} v \notin U \Longrightarrow d(T^{2j f(n) + i_j + 1} v, y) > \rho \geq 3\eta$. 
Then, $d(T^{2jf(n) + i_j + 1} z', y) > 2\eta$, and so $d(T^{2jf(n) + i_j + 1} z, T^{2jf(n) + i_j + 1} z') > \eta$, completing the proof that $Z$ is 
$(n, \eta)$-separated. Then,
\begin{multline*}%\label{finalbd2}
\sum_{x \in Z} e^{S_n \phi(x)} = \sum_{j = 1}^{\lfloor (n-f(n))/2f(n) \rfloor} \sum_{(v,w) \in U(j)} e^{S_n \phi(z(j,v,w))}\\
\stackrel{(\ref{partbd3.5})}{\geq} \frac{n - f(n)}{2f(n)^3} C^2 e^{nP(X,T,\phi)} n^{-2/3}. 
\end{multline*}

Again we use Lemma~\ref{sepscaleindep} to pass to $Z' \subseteq Z$ which is $(n, \delta)$-separated and for which 
\[
\sum_{x \in Z'} e^{S_n \phi(x)} \geq \frac{n - f(n)}{2f(n)^3} C^2 M(\eta)^{-1} e^{nP(X,T,\phi)} n^{-2/3},
\]
implying that
\[
Z(X,T,\phi,n,\delta) \geq \frac{n - f(n)}{2f(n)^3} C^2 M(\eta)^{-1} e^{nP(X,T,\phi)} n^{-2/3}.
\]

However, this will contradict Corollary~\ref{speccor} or \ref{transcor} for large enough $k$. Therefore, our assumption was incorrect and $\mu$ is fully supported.

\end{proof}

%\begin{remark} 
%The same techniques used in the proofs of Theorems~\ref{langbd2} and Theorem~\ref{mainthm2} would allow a version of Theorem~\ref{mainthm} assuming a weaker version of non-uniform two-sided specification, where one only assumes that any finite set of words can be combined for a single choice of gaps which are each less than or equal to the corresponding thresholds for each consecutive pair of words. This would cause an extra factor of $f(n)$ in the upper bound from Theorem~\ref{langbd} and of $(f(n))^3$ in the denominator in the right-hand side of (\ref{toomanywords}), which do not affect the proof. We did not explicitly state this result because we did not want to add to the already large list of subtly different properties considered in this paper.
%\end{remark}

\section{Examples}\label{examples}

Here we present some examples of $(X,T)$ and $\phi$ satisfying our hypotheses for which we believe our results to give the first proof of uniqueness of equilibrium state. We begin with $(X,T)$ with weakened specification properties. The following class of subshifts is defined in \cite{stanley}, which as usual are endowed with $T$ the left shift map.

\begin{example}
Given any alphabet $A = \{0,1,\ldots,k\}$ and increasing subadditive $h: \mathbb{N} \rightarrow \mathbb{N}$, the \textbf{bounded density shift} associated to $k$ and $h$, denoted $X_{k,h}$, is the set of all $x \in A^{\mathbb{Z}}$ so that for all $i \in \mathbb{Z}$ and $n \in \mathbb{N}$, $x(i) + \ldots + x(i + n - 1) \leq h(n)$.
\end{example}

By subadditivity, for any bounded density shift, $h(n)/n$ approaches some constant $\alpha$ (called the gradient), and $h(n) \geq n\alpha$ for all $n$. It was shown in \cite{stanley} that $X_{k,h}$ has specification if and only if $\alpha > 0$ and $h(n) - n\alpha$ is bounded.

\begin{theorem}\label{bdspec}
If $h(n) = n \alpha + e(n)$, where $e(n)$ is increasing, then $(X_{k,h},T)$ has non-uniform specification with gap bounds $f(n) := 2e(n)/\alpha$.
\end{theorem}

\begin{proof}
We first recall that since $(X_{k,h}, T)$ is a subshift, it is expansive for a constant $\delta$ where $d(x,y) \leq \delta \Longrightarrow x(0) = y(0)$. This means that $\delta$-shadowing any $x$ for $n$ iterates is the same as agreeing with $x$ for $n$ letters. This means that the claimed non-uniform specification is implied by the following: for any $w_1, \ldots, w_k \in \mathcal{L}(X_{k,h})$, with lengths $n_1, \ldots, n_k$, and for any $m_1$, $\ldots$, $m_{k-1}$ with $m_i \geq \max(2e(n_i)/\alpha, 2e(n_{i+1})/\alpha) \newline \geq (e(n_i) + e(n_{i+1}))/\alpha$, the word $w = w_1 0^{m_1} w_2 \ldots 0^{m_{k-1}} w_k$ is in $\mathcal{L}(X_{k,h})$.  

Consider any such $(w_i)$, $(n_i)$, and $(m_i)$. It suffices to show that for every subword $v$ of $w$, the sum of the letters of $v$ is less than or equal to $h(|v|)$. Since $h$ is nondecreasing, it clearly suffices to consider only the case where $v$ neither begins nor ends with a subword of some $0^{m_i}$. We can then write $v$ as 
\[
v = s 0^{m_i} w_{i+1} \ldots w_j 0^{m_j} p,
\]
where $1 < i \leq j < k-1$, $s$ is a suffix of $m_{i-1}$ (say of length $a$), and $p$ is a prefix of $m_{j+1}$ (say of length $b$). Since each $w_i$ was in $\mathcal{L}(X_{k,h})$, the sum of the letters of any $w_i$ is at most $h(n_i)$. The sum of the letters of $v$ is then less than or equal to
$h(a) + h(n_i) + \ldots + h(n_j) + h(b)$. Also,
\begin{multline*}
h(|v|) = h(a + m_i + n_i + \ldots + n_j + m_j + b) \geq \alpha(a + m_i + n_i + \ldots + n_j + m_j + b)
\geq\\
\alpha (a + n_i + \ldots + n_j + b) + e(a) + e(n_i) + \ldots + e(n_j) + e(b) \geq h(a) + h(n_i) + \ldots + h(n_j) + h(b).
\end{multline*}
Therefore, the sum of the letters of $v$ is less than or equal to $h(|v|)$. Since $v$ was arbitrary, $w$ is in $\mathcal{L}(X_{k,h})$, completing the proof.
\end{proof}

We do not believe that uniqueness of measure of maximal entropy is known for any bounded density shift without specification. Theorem~\ref{mainthm}, however, yields the following corollary (by taking $\phi = 0$).

\begin{corollary}
Any bounded density shift $(X_{k,h},T)$ with $h(n) = n\alpha + e(n)$ for $e(n)$ nondecreasing with $\liminf_{n \rightarrow \infty} e(n)/\ln n = 0$ has a unique measure of maximal entropy (which is fully supported and has the K-property). 
\end{corollary}

We also present a class of subshifts with non-uniform transitivity (but not non-uniform specification) to which our results apply. Interestingly, these subshifts cannot have periodic points, and yet our results imply uniqueness of the equilibrium state.

\begin{example}\label{transex}
Fix any Sturmian subshift $S$ (see Chapter 6 of \cite{fogg} for an introduction to Sturmian subshifts) and sequence of integers $\{n_k\}$ where $n_k \geq 2n_{k-1} + 2k$ for every $k$. Define the associated subshift $X_{S, \{n_k\}}$ as the set of all $x \in \{0,1\}^{\mathbb{Z}}$ so that for every $i \in \mathbb{Z}$ and $k \in \mathbb{N}$, the word $x(i) \ldots x(i + n_k + 2k - 1)$ contains a $k$-letter word in the language of $S$. 
\end{example}

\begin{remark}
Note that by definition, for any $x \in X_{S, \{n_k\}}$, the orbit closure of $x$ contains a point of $S$. Since Sturmian shifts contain no periodic points, this means that $X_{S, \{n_k\}}$ contains no periodic points. We also show that $X_{S, \{n_k\}}$ cannot have non-uniform specification. Suppose for a contradiction that $X_{S, \{n_k\}}$ has non-uniform specification with gap bounds $f(n)$. Choose any $x \in X_{S, \{n_k\}}$ with $x(0) = 0$. Then by taking limits of points which $\delta$-shadow $(x,\ldots,x)$ for $(1,\ldots,1)$ iterates with gaps $(f(1),\ldots,f(1))$, we see that $X_{S, \{n_k\}}$ must contain a sequence $y$ of the form $\ldots 0 w_{-1} 0 w_0 0 w_1 0 \ldots$, i.e. $y(m(1+f(1))) = 0$ for all $m \in \mathbb{Z}$. Then the orbit closure of $y$ contains some $s \in S$, which also has the property that $s(m(1+f(1))) = 0$ for all $m \in \mathbb{Z}$. However, this is impossible; Sturmian shifts have a unique invariant measure with respect to which all powers are ergodic, and so the existence of $s$ would contradict the ergodic theorem.
\end{remark}

\begin{theorem}
If $\lim_{n \rightarrow \infty} \frac{\ln n_k}{k} = \infty$, then $(X_{S, \{n_k\}}, T)$ has non-uniform transitivity with gap bounds $f(n)$ satisfying $\lim_{n \rightarrow \infty} \frac{f(n)}{\ln n} = 0$.
\end{theorem}

\begin{proof}
Consider any such $S$, $\{n_k\}$, and associated subshift $X = X_{S, \{n_k\}}$. We claim that $(X, T)$ has non-uniform transitivity with gap bounds $f(n)$ where $f(n) = 2k$ for the minimal $k$ where $n \leq n_k$. This implies the desired result, since clearly $\lim_{n \rightarrow \infty} \frac{f(n)}{\ln n} = 0$ if $\lim_{n \rightarrow \infty} \frac{\ln n_k}{k} = \infty$. As above, $\delta$-shadowing orbit segments is just the same as containing words from the language, so it suffices to show that 
for any $v,w \in \mathcal{L}(X)$ with $|v| = |w| \leq n_k$, there exists $u$ with $|u| = 2k$ so that $vuw \in \mathcal{L}(X)$.

Since $v$ and $w$ can be extended on the left and right to words of length $n_k$ in $\mathcal{L}(X)$, it suffices to treat only the case $|v| = |w| = n_k$; choose any such $v,w$. Of the prefixes of $v$, choose the one with maximal length which is in $\mathcal{L}(S)$, and denote it by $p_v$. Similarly define a prefix $p_w$ of $w$, and suffixes $s_v$ and $s_w$ of $v$ and $w$ respectively. Since $p_v, s_w \in \mathcal{L}(S)$, there exist a left-infinite sequence $x$ and a right-infinite sequence $y$ for which $x p_v, s_w y \in \mathcal{L}(S)$. Similarly, since $s_v, p_w \in \mathcal{L}(S)$, there exist $s, t$ with length $k$ so that $s_v s, t p_w \in \mathcal{L}(S)$. We claim that $x v s t w y \in X$, which will imply that $v st w \in \mathcal{L}(X)$, completing our proof by taking $u = st$. 

For this proof, we need to show that for every $j$ and every $(n_j + 2j)$-letter subword $z$ of $xvstwy$, $z$ contains a word in $\mathcal{L}(S)$. We break into cases, and first treat the case where $j > k$. Then $z$ has length $n_j + 2j > 2n_{j-1} + 4j \geq 2n_k + 2k + 2j$. Then $z$ must contain a $j$-letter subword of either $x$ or $y$, which by definition is in $\mathcal{L}(S)$.

Suppose instead that $j \leq k$. If $z$ contains letters from both $s$ and $t$, then it contains a $j$-letter subword of one of them, which is in $\mathcal{L}(S)$. The remaining case is where $z$ is a subword of either $xvs$ or $twy$; without loss of generality, we assume the latter. If $j = k$, then since $|z| = 2n_k + 2k > n_k + 2k$, $z$ contains either a $j$-letter subword of $t$ or $y$, which is in $\mathcal{L}(S)$. So, we from now on assume $j < k$. This means that we can write $z = qr$, where either $q$ is a suffix of $t$ and $r$ is a prefix of $w$ or $q$ is a prefix of $w$ and $r$ is a prefix of $y$; without loss of generality, we assume the former. 

If $|q| + |p_w| \geq j$, then the $j$-letter prefix of $z = qr$ is a subword of $t p_w \in \mathcal{L}(S)$, so it would be in $\mathcal{L}(S)$ as well. The only remaining case is $|q| + |p_w| < j$. We claim that here, $r$ contains a $j$-letter word in $\mathcal{L}(S)$. To see this, recall that $w$ was in $\mathcal{L}(X)$, and so there exists $q'$ with $|q'| = |q|$ so that $q' w \in \mathcal{L}(X)$. In particular, this means that $q'r$, which is a subword of $q' w$ with length $n_j + 2j$, contains a $j$-letter word in $\mathcal{L}(S)$. If this word was not entirely contained in $r$, then it would contain a prefix of $w$ of length $|p_w| + 1$, which would be in $\mathcal{L}(S)$, contradicting maximality of $p_w$ in its definition. So, we know that $r$ contains a $j$-letter word in $\mathcal{L}(S)$, implying that $z = qr$ does as well. This shows that $vuw \in \mathcal{L}(X)$ (for $u = st$), completing the proof.

\end{proof}

Finally, we present a simple condition on a potential $\phi$ which guarantees slowly growing partial sum variation bounds, in the spirit of a proposition from \cite{bowen}. The proof is essentially identical to the one given in \cite{bowen}, and so we omit it here.

\begin{theorem}\label{gbounds}
For $(X,T)$, a potential $\phi$, $n \in \mathbb{N}$, and $\eta > 0$, define $\textrm{Var}(X,T,\phi,n,\eta)$ to be the maximum of $|\phi(x) - \phi(y)|$ over pairs $(x,y)$ where $d(T^i x, T^i y) < \eta$ for all $|i| \leq n$. Then $\phi$ has partial sum variation bounds $g(n)$ at scale $\eta$ defined by
\[
g(n) = 2\sum_{i=0}^{\lfloor n/2 \rfloor} \textrm{Var}(X,T,\phi,n,\eta).
\]
\end{theorem}

The following is an immediate corollary of Theorem~\ref{gbounds} and Lemma~\ref{sumscaleindep}.

\begin{corollary}\label{slowgrowth}
If $(X,T)$ is expansive, $\phi$ is a potential, $\eta > 0$, and $\phi$ is a potential with $\lim_{n \rightarrow \infty} n \textrm{Var}(X,T,\phi,n,\eta) = 0$, then $\phi$ has partial sum variation bounds $g(n)$ satisfying $\lim_{n \rightarrow \infty} g(n)/\ln n = 0$.
\end{corollary}

It is simple to construct potentials for which $\textrm{Var}(X,T,\phi,n,\eta)$ grows as slowly as desired; the following is one example.

\begin{example}\label{potex}
For any increasing $h: \mathbb{N} \rightarrow \mathbb{R}^+$ and the full shift $(X,T)$ on symbols $0$ and $1$, define a potential $\phi_h$ by $\phi_h(x) = \frac{1}{h(k)}$, where $k$ is the maximal integer where $x(-k) = \ldots = x(k)$. (If $x$ consists entirely of $0$s or $1$s, then $\phi(x) = 0$.)
\end{example}

\begin{lemma}
If $\lim_{n \rightarrow \infty} h(n)/n = \infty$, then $\phi_h$ satisfies the hypotheses of Corollary~\ref{slowgrowth}. If $\frac{1}{h(n)}$ is not summable, then $\phi_h$ is not Bowen.
\end{lemma}

\begin{proof}
We claim that $\textrm{Var}(X,T,\phi_h,n,\delta) = \frac{1}{h(n)}$. To see this, first note that a pair $(x,y)$ satisfies $d(T^i x, T^i y) < \delta$ for $|i| \leq n$ iff $x(-n) \ldots x(n) = y(-n) \ldots y(n)$. If $\phi_h(x) = \frac{1}{h(i)}$ for some $i < n$, then $x(-i) = \ldots = x(i)$ and either $x(-(i+1))$ or $x(i+1)$ is not equal to $x(0)$. The same is then true of $y$, so $\phi_h(y) = \phi_h(x)$.

Therefore, if $\phi_h(x) \neq \phi_h(y)$, then both are less than or equal to $\frac{1}{h(n)}$, so $|\phi_h(x) - \phi_h(y)| \leq \frac{1}{h(n)}$, implying $\textrm{Var}(X,T,\phi_h,n,\delta) \leq \frac{1}{h(n)}$. Finally, $x = 0^{\mathbb{Z}}$ and $y$ defined by $y(i) = 0$ iff $|i| \leq n$ have $|\phi_h(x) - \phi_h(y)| = \frac{1}{h(n)}$, so $\textrm{Var}(X,T,\phi_h,n,\delta) = \frac{1}{h(n)}$.

Since $\lim_{n \rightarrow \infty} h(n)/n = \infty$, $\lim_{n \rightarrow \infty} n \textrm{Var}(X,T,\phi_h,n,\delta) = 0$, and so $\phi_h$ satisfies the hypotheses of Corollary~\ref{slowgrowth}.

Similarly to above, take $x = 0^{\mathbb{Z}}$ and $y$ defined by $y(i) = 0$ iff $i \geq 0$. Then
$d(T^i x, T^i y) < \delta$ for $0 \leq i < n$, and $\left| S_n \phi_h(x) - S_n \phi_h(y) \right| 
= \sum_{i = 1}^n \frac{1}{h(i)}$. It is then clear that if $\frac{1}{h(n)}$ is not summable, then $\phi_h$ does not have the Bowen property, completing the proof.

\end{proof}

Theorems~\ref{mainthm} and \ref{mainthm2} then provide uniqueness of equilibrium state, its full support, and sometimes the $K$-property, for many of these examples, including measures of maximal entropy for various bounded density shifts and the shifts $X_{S, \{n_k\}}$ of Example~\ref{transex}. However, the easiest new application is probably the uniqueness of equilibrium state for any expansive $(X,T)$ with weak specification and any $\phi$ satisfying Corollary~\ref{slowgrowth}. 

This even includes examples on manifolds. For instance, if $X = [0,1)$ and $T: x \mapsto 2x \pmod 1$, then though $(X,T)$ is non-invertible, its natural extension is invertible and has weak specification (for $f = 0$). A simple example of $\phi$ on $(X,T)$ which is not Bowen but has unique equilibrium state by Corollary~\ref{slowgrowth} is $\phi(x) = \frac{1}{1 + \log(1/x) \log \log(1/x)}$.

\bibliographystyle{plain}
\bibliography{gapspec}

\end{document}